\documentclass[11pt]{amsart}

\usepackage[utf8]{inputenc}
\usepackage[english]{babel}
\usepackage[margin=1in]{geometry}
\usepackage{microtype}
\usepackage{amsmath,amsthm,amssymb,mathrsfs}
\usepackage{mathtools}
\usepackage{graphicx,subcaption}
\usepackage{hyperref}
\usepackage{xcolor}
\hypersetup{
    colorlinks=true,    
    linkcolor=blue,          
    citecolor=blue,      
    filecolor=blue,      
    urlcolor=blue}
\usepackage{tikz-cd,adjustbox}
\usepackage{spverbatim,comment}
\usepackage[shortlabels]{enumitem}
\usepackage{tabularx}
\usepackage{colonequals}

\usepackage{pgfplots}
\pgfplotsset{compat=1.18}
\usepackage{pgfplotstable}
\usepackage{tikz}
\usetikzlibrary{3d,calc,intersections}

\theoremstyle{plain}
\newtheorem{theorem}{Theorem}[section]

\newtheorem{proposition}[theorem]{Proposition}
\newtheorem{corollary}[theorem]{Corollary}
\theoremstyle{definition} 
\newtheorem{definition}[theorem]{Definition}
\newtheorem{remark}[theorem]{Remark}

\newtheorem{example}[theorem]{Example}

\newcommand{\Rom}[1]{\uppercase\expandafter{\romannumeral#1}}

\newcommand{\E}{\mathcal E}

\renewcommand{\L}{\mathcal L}

\newcommand{\cO}{\mathcal O}

\newcommand{\X}{\mathcal X}
\newcommand{\Y}{\mathcal Y}
\newcommand{\Z}{\mathcal Z}

\newcommand{\CC}{\mathbb C}

\newcommand{\FF}{\mathbb F}

\newcommand{\PP}{\mathbb P}
\newcommand{\QQ}{\mathbb Q}
\newcommand{\RR}{\mathbb R}

\newcommand{\ZZ}{\mathbb Z}

\DeclareMathOperator{\Bat}{Bat}

\DeclareMathOperator{\Hom}{Hom}

\newcommand*{\sHom}{\mathscr{H}\kern -.5pt om}

\DeclareMathOperator{\Cone}{Cone}

\DeclareMathOperator{\Pic}{Pic}

\DeclareMathOperator{\PC}{PC}
\DeclareMathOperator{\rays}{G} 
\DeclareMathOperator{\Conv}{Conv} 
\DeclareMathOperator{\NE}{NE} 

\newcommand{\ie}{\emph{i.e.,} }
\newcommand{\eg}{\emph{e.g.,} }

\newcommand{\resp}{\emph{resp.} }

\DeclarePairedDelimiter\abs{\lvert}{\rvert}

\makeatletter
\let\oldabs\abs
\def\abs{\@ifstar{\oldabs}{\oldabs*}}
\makeatother

\newcommand{\thistheoremname}{}
\newtheorem*{genericthm}{\thistheoremname}

\author{Zhengning Hu}
\address{Department of Mathematics, University of Arizona, Tucson, AZ 85721}
\email{zhengninghu@arizona.edu}

\author{Rohan Joshi}
\address{UCLA Department of Mathematics, Los Angeles, CA 90095-1555}
\email{rohansjoshi@math.ucla.edu}


\title{Classifying weak Fano toric varieties of Picard rank $3$}

\begin{document}

\begin{abstract}
	We provide a systematic method to classify all smooth weak Fano toric varieties of Picard rank $3$ in any dimension using Macaulay2, and describe the classification explicitly in dimensions $3$ and $4$. There are $28$ and $114$ isomorphism classes of rank $3$ weak Fano toric threefolds and fourfolds, respectively.
\end{abstract}
\maketitle

\section{Introduction}

Toric varieties are algebraic varieties that contain an algebraic torus as an open dense subset such that the action of the torus on itself extends to the whole variety. 
Normal toric varieties can be described by combinatorial data, specifically by a fan sitting in a real vector space generated by a lattice. 
Normal toric varieties admit a computational theory, where geometric properties, topological invariants, sheaf cohomology, etc. can be computed from this combinatorial data. 
Many of these methods are implemented in the \texttt{NormalToricVarieties} package \cite{NormalToricVarietiesSource} in computer algebra system Macaulay2 \cite{M2}. 

A Fano (\resp weak Fano) variety is a smooth projective algebraic variety $X$ whose anticanonical divisor $-K_X$ is ample (\resp nef and big). 
There are finitely many isomorphism classes of weak Fano toric varieties in any dimension. 
In this paper we consider the problem of classifying weak Fano toric varieties of Picard rank $3$ up to isomorphism. 
By work of Batyrev \cite{Bat:91}, any smooth projective toric variety of Picard rank $3$ over an algebraically closed field either has three primitive collections, in which case it is a toric projective bundle, or has five primitive collections, in which case it can be generated by a combinatorial recipe described by Batyrev. 
At the 2024 Macaulay2 workshop in Salt Lake City, the authors and other members of the Toric Varieties working group wrote functions for the \texttt{NormalToricVarieties} package that construct toric projective bundles as well as implement Batyrev's recipe. 
We build on this work and prove bounds on the inputs to these functions such that all weak Fano outputs can be generated by bounded inputs.
We write Macaulay2 functions for exhaustively constructing all weak Fano toric varieties of Picard rank $3$ in any dimension and computing their numerical invariants. Using this, we describe explicitly all isomorphism classes of weak Fano toric varieties of Picard rank $3$ in dimensions $3$ and $4$: see Appendices~\ref{app:3Picard3Fold} and~\ref{app:3Picard4Fold}.

In addition we explore the flops and (regular) blowdowns of these varieties. We write Macaulay2 functions which implement possible equivariant blowups and blowdowns based on work of Sato \cite[Section 4]{Sat:00}. 
Using these functions, we explicitly describe possible flops and blowdowns of toric varieties in our lists: see the column ``Blowdown and flop'' in Appendices~\ref{app:2Picard3Fold}, \ref{app:2Picard4Fold}, \ref{app:3Picard3Fold} and \ref{app:3Picard4Fold}. 
We also prove a converse to Bott-Steenbrink-Danilov vanishing \cite[Theorem 9.3.1]{CLS:11}, characterizing Fano varieties among weak Fano toric varieties by the vanishing of a cohomology group.

In this paper we work over an algebraically closed field $k$ of arbitrary characteristic.

\subsection*{Prior work} 
The classification of Fano threefolds was completed by \cite{Isk:77, Isk:78, MoMu:86}: a description of all families of Fano threefolds is available on the website Fanography \cite{fanography} (throughout this paper, Fano and weak Fano varieties are understood to be smooth and projective).
Fano toric varieties are classified up to dimension six, by Batyrev \cite{Bat:99}, Sato \cite{Sat:00}, and {\O}bro \cite{Obr:07}. 
This classification is available online in the Graded Rings Database and is also incorporated into \href{https://macaulay2.com/doc/Macaulay2/share/doc/Macaulay2/NormalToricVarieties/html/_smooth__Fano__Toric__Variety_lp__Z__Z_cm__Z__Z_rp.html}{Macaulay2}. 
{\O}bro claims in his thesis to have completed the classification up to dimension eight \cite{ObroThesis}. 
Weak del Pezzo toric surfaces are classified in \cite{Sat:02}. 
Moreover, \cite{Sat:02} also classifies \textit{weakened} toric Fano threefolds: a weakened Fano variety is a weak Fano variety that is not Fano but can be deformed to a Fano variety under a small deformation. 
Finally, weak Fano (not necessarily toric) threefolds of Picard rank $2$ are classified except for a few remaining open cases in \cite{CuMa:13} and \cite{CuMa:24}.

\subsection*{Acknowledgments} 
We would like to thank Gregory Smith, Thiago Holleben, and Burt Totaro for many helpful conversations. We would also like to thank the organizers of the 2024 Macaulay2 workshop in Salt Lake City. The second author is supported by NSF Graduate Research Fellowship Grant No. DGE-2034835.

\section{Primitive collections and primitive relations}

Let $M \cong \ZZ^d$ be the character lattice of a $d$-dimensional torus $T$ with dual lattice $N = \Hom_\ZZ(M,\ZZ)$.
We associate to a rational polyhedral fan $\Sigma \subseteq N_\RR = N \otimes_\ZZ \RR$ the toric variety $X = X(\Sigma)$.
For every integer $1 \leq l \leq d$, we denote by $\Sigma(l)$ the set of cones in $\Sigma$ of dimension $l$, and $\rays(\Sigma)$ or $\rays(X)$ the set of primitive generators of rays of $\Sigma$.

In this paper we shall consider smooth and proper (or projective) toric varieties $X = X(\Sigma)$ over $k$ where $\Sigma$ is a regular and complete (or projective) fan in the vector space $N_\RR$.
We first review in this section some useful definitions and results we are going to use regarding primitive collections and primitive relations.
More details can be found in \cite[Chapter 6]{CLS:11}.

\begin{definition}\label{defn:primitive}
	Let $X = X(\Sigma)$ be a toric variety with $\Sigma \subseteq N_\RR$ its associated regular and complete fan.
    We call $P = \{x_1, \cdots, x_n \}  \subseteq \rays(\Sigma)$ a \textit{primitive collection} of $\Sigma$ if $P$ does not generate a cone of $\Sigma$ but every proper subset of $P$ does. 
    We denote the set of primitive collections of $X$ by $\PC(X)$. 
    Since $\Sigma$ is complete, the vector $x_1 + \cdots + x_n$ is in the relative interior of a unique cone $\sigma(P)$. 
    Let $y_1, \cdots, y_m$ be the primitive generators of $\sigma(P)$. 
    Then $x_1 + \cdots + x_n$ can be uniquely written as a positive $\ZZ$-linear combination of $y_1, \cdots, y_m$, that is 
    \[x_1 + \cdots + x_n = a_1 y_1 + \cdots + a_m y_m.\]  
This linear relation among elements of $\rays(\Sigma)$ is called the \textit{primitive relation} associated to the primitive collection $P$. 
We can consider $x_1 + \cdots + x_n - (a_1 y_1 + \cdots + a_m y_m)$ as an element $r(P) \in N^1(X)^\vee = N_1(X)$, the group of $1$-cycles on $X$ up to numerical equivalence. 
We define the \textit{degree} of $P$ by
\[\deg P \colonequals (-K_X) \cdot r(P) = n - (a_1 + \cdots + a_m).\] 
\end{definition}

Since a projective toric variety $X$ is log Fano (\ie there exists an effective $\QQ$-divisor $D$ on $X$ such that the pair $(X,D)$ is klt and $-(K_X+D)$ is ample), the Cone Theorem implies the Mori cone (the closure of cone of curves in $N_1(X)$) of $X$, denoted by $\overline{\NE}(X)$, is rational polyhedral. 
This can also be shown by more elementary means.
As stated in \cite[Proposition 1.6]{Rei:83}, if $X$ is a proper toric variety, every irreducible curve on $X$ contained in a smallest stratum can be deformed in the stratum to a union of curves, each of which is contained in a lower dimensional substratum, by a $1$-parameter subgroup of the big torus with points of indeterminacy resolved.
Thus every effective $1$-cycle on $X$ is rationally (and numerically) equivalent to a positive linear combination of some torus-invariant strata.
Since the cone of curves of $X$ is spanned by torus-invariant curves which is a finite set, its Mori cone $\overline{\NE}(X)$ is rational polyhedral.

The following proposition gives an explicit description of the Mori cone in terms of primitive relations.

\begin{theorem}\cite[Theorem 6.4.11]{CLS:11}\label{thm:toricconethm}
Let $X = X(\Sigma)$ where $\Sigma$ is a projective regular fan. Then the Mori cone $\overline{\NE}(X)$ is generated by the elements $r(P) \in N_1(X)$ for all primitive collections $P$ of $\Sigma$.
\end{theorem}

It is worth noting that not every primitive relation is contained in an extremal ray of the Mori cone, for example, in Section~\ref{sec:batyrevconstruction}, the Batyrev varieties have five primitive collections, but only three of them are contained in extremal rays.
We illustrate this fact by an example below.

\begin{example}\label{ex:fanobatyrevsurface}
Consider $X = \text{Bl}_1(\PP^1 \times \PP^1)$, \ie blowup of $\PP^1 \times \PP^1$ at a torus-invariant point.
Such $X$ can be constructed as the toric variety associated to the fan in $\RR^2$ with rays $x_1 = e_1$, $x_2 = e_2$, $x_3 = -e_1$, $x_4 = -e_1 - e_2$ and $x_5 = -e_2$.

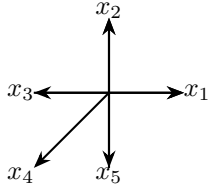
\begin{center}
\begin{tikzpicture}[scale=1, >=Stealth]

  \draw[->, thick] (-1,0) -- (1,0) node[right]{};
  \draw[->, thick] (0,-1) -- (0,1) node[above]{};

  \draw[->, thick] (0,0) -- (1,0) node[midway, right=10pt] {\small $x_1$};
  \draw[->, thick] (0,0) -- (0,1) node[midway, above=10pt] {\small $x_2$};
  \draw[->, thick] (0,0) -- (-1,0) node[midway, left=10pt] {\small $x_3$};
  \draw[->, thick] (0,0) -- (-1,-1) node[midway, below left=10pt] {\small $x_4$};
  \draw[->, thick] (0,0) -- (0,-1) node[midway, below=10pt] {\small $x_5$};

\end{tikzpicture}
\end{center}

Then every pair of nonadjacent rays corresponds to a primitive collection. The primitive relations are $x_1+x_3 = 0$, $x_1+x_4=x_5$, $x_2+x_4=x_3$, $x_2+x_5=0$ and $x_3+x_5=x_4$. $X$ is a Batyrev variety (see Section \ref{sec:batyrevconstruction}): it is a smooth projective toric variety of Picard rank $3$ with five primitive collections. 
In fact it is easy to check it is Fano using the criterion (Corollary~\ref{cor:wFanocriterion}) below.

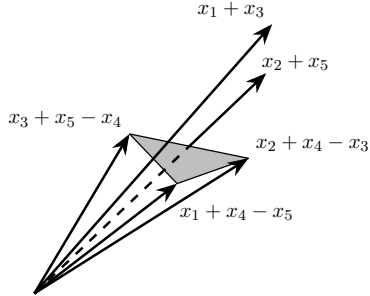
\begin{figure}[h!]
\begin{tikzpicture}[scale=1.2, >=Stealth]
  \begin{axis}[
    view={135}{25},         
    axis lines=none,
    ticks=none,
    enlargelimits=0.15
  ]

    \coordinate (A) at (2,1,0); 
    \coordinate (B) at (1,2,1);
    \coordinate (C) at (2,2,1.5);
    \coordinate (O) at (0,0,0);  
    \coordinate (S) at (0,1,1.5);
    \coordinate (T) at (1.25, 2.5, 1.87);

    \addplot3[patch,patch type=triangle,opacity=0,faceted color=black] coordinates {
      (2,1,0)
      (1,2,1)
      (2,2,1.5)
    };

    \addplot3[patch,patch type=triangle,opacity=0,faceted color=black] coordinates {
      (0,0,0)
      (2,1,0)
      (1,2,1)
    };

    \addplot3[patch,patch type=triangle,opacity=1,color=black!25,faceted color=black] coordinates {
      (0,0,0)
      (1,2,1)
      (2,2,1.5)
    };

    \addplot3[patch,patch type=triangle,opacity=0,faceted color=black] coordinates {
      (0,0,0)
      (2,2,1.5)
      (2,1,0)
    };

    \addplot3[->, thick] coordinates {(2,1,0) (0,0,0)};
    \addplot3[->, thick] coordinates {(2,1,0) (1,2,1)};
    \addplot3[->, thick] coordinates {(2,1,0) (2,2,1.5)};
    \addplot3[->, thick] coordinates {(2,1,0) (0,1,1.5)};

    \draw[name path = line2,opacity=0] (2,1,0) -- (1.25, 2.5, 1.87);
    \draw[name path = line1,opacity=0] (2,2,1.5) -- (1, 2, 1);
    \path[name intersections={of=line1 and line2,by = P}];
    
    \addplot3[->, thick,dashed,opacity=0] coordinates {(2,1,0) (1.25, 2.5, 1.87)};
    \draw[dashed, thick] (2,1,0) -- (P);
    \draw[->, thick] (P) -- (1.25, 2.5, 1.87);
    

    \node[scale=0.6] at (axis cs:0.2,0.2,0) [anchor=north west] {$x_1+x_4-x_5$};
    \node[scale=0.6] at (axis cs:1,2,1) [anchor=south west] {$x_2+x_4-x_3$};
    \node[scale=0.6] at (axis cs:2,2,1.5) [anchor=south east] {$x_3+x_5-x_4$};
    \node[scale=0.6] at (axis cs:0,1,1.5) [anchor=south east] {$x_1+x_3$};
    \node[scale=0.6] at (axis cs:1.35,2.5,1.87) [anchor=south west] {$x_2+x_5$};
  \end{axis}
\end{tikzpicture}
\caption{The Mori cone $\overline{\NE}(X)$}
\label{fig:moricone}
\end{figure}

The cone of curves (see Figure~\ref{fig:moricone}) is a cone over a triangle, with the three extremal rays generated by the primitive relations $x_1+x_4-x_5$, $x_2+x_4-x_3$ and $x_3+x_5-x_4$. 
The other two primitive relations span rays that lie on facets of the cone. 

\end{example}

By Theorem~\ref{thm:toricconethm}, we have the following criterion for a smooth projective toric variety to be Fano or weak Fano.

\begin{corollary}\label{cor:wFanocriterion}
	A smooth projective toric variety $X(\Sigma)$ is Fano (\resp weak Fano) if and only if for all primitive collections $P \in \PC(X)$, $\deg P > 0$ (\resp $\geq 0$). 
\end{corollary}

Let $Y$ be a smooth toric variety and consider a collection of line bundles $\mathcal{O}(D_0), \cdots, \mathcal{O}(D_k)$ on $Y$. 
Since the divisor class group of a toric variety is generated by torus-invariant divisors, we may assume that the $D_i$'s are torus-invariant. 
Then we can take the projectivization of the decomposable vector bundle $\cO(D_0) \oplus \cdots \oplus \cO(D_k)$ which is a smooth toric variety $X = \PP_Y(\cO(D_0) \oplus \cdots \oplus \cO(D_k))$. Since tensoring a vector bundle by a line bundle doesn't change the total space of its projectivization, we may assume $D_0 = 0$. 
We will call varieties constructed in this way \textit{toric projective bundles}. We can use primitive relations to state a criterion for a toric variety to be a toric projective bundle.

\begin{proposition}\cite[Proposition 4.1]{Bat:91}
The toric variety $X = X(\Sigma)$ is a toric $\PP^k$-bundle if and only if there exists a primitive collection $P = \{ x_0, \cdots, x_k \} \in \PC(X)$ such that $x_0 + \cdots + x_k = 0$,
and $P$ is disjoint from any other primitive collection.
\end{proposition}

\begin{definition}
	A fan is a \textit{splitting fan} if any two distinct primitive collections are disjoint.
\end{definition}

\begin{proposition}\cite[Theorem 4.3]{Bat:91}
	Let $\Sigma$ be a splitting fan. Then the corresponding toric variety $X(\Sigma)$ is a projectivization of a decomposable bundle over a toric variety associated to a splitting fan of a smaller dimension.
\end{proposition}

Varieties of the form above which are iterated toric projective bundles are sometimes called ``generalized Bott towers'' in the literature.

Kleinschmidt and Sturmfels showed that a smooth proper toric variety of Picard rank $3$ is always projective \cite{KlSt:91}. Batyrev proved a strong theorem classifying smooth projective varieties of Picard rank $3$ into two types, based on the number of primitive collections.

\begin{theorem}\cite[Theorem 5.7]{Bat:91}\label{thm:rank3classification}
	A smooth projective toric variety $X$ of Picard rank $3$ has either three or five primitive collections. 
    If $X$ has three primitive collections, then $\Sigma$ is a splitting fan, so $X = \PP_Y(\cO \oplus \cO(D_1) \oplus \cdots \oplus \cO(D_k))$ for some toric divisors $D_1,\cdots,D_k \in \Pic(Y)$, where $Y$ is a smooth projective toric variety of Picard rank $2$.
\end{theorem}

When $X$ has five primitive collections, it can be constructed using a combinatorial recipe described in \cite[Theorem 6.6]{Bat:91}. 
In order to pursue our classification of weak Fano toric varieties of Picard rank $3$, we will consider the cases of three and five primitive collections separately in the next two sections. 

\section{Projective bundles}\label{sec:projbundle}

Let $X$ be a weak Fano toric variety of Picard rank 3 with three primitive collections. 
According to Theorem~\ref{thm:rank3classification}, we know there exists a smooth projective toric variety $Y$ and toric divisors $D_1, \cdots, D_k$ on $Y$ such that $X \cong \PP_Y(\cO \oplus \cO(D_1) \oplus \cdots \oplus \cO(D_k))$. 
In fact, the toric variety $Y$ must also be weak Fano by the following proposition.

\begin{proposition}\label{prop:baseWeakFano}
	Let $\E$ be a decomposable vector bundle on a projective toric variety $X$ so that the projectivization $\PP_X(\E)$ is a toric projective bundle. If $\PP_X(\E)$ is Fano (\resp weak Fano), then $X$ is also Fano (\resp weak Fano).
\end{proposition}
\begin{proof}	
The Fano case is proved in \cite[Theorem 1.6]{SzWi:90}. Now let's assume $\PP_X(\E)$ is weak Fano.
We denote by $p \colon \PP_X(\E) \to X$ the projection.
Since $\PP_X(\E)$ is a projective toric variety, the discussion after Definition~\ref{defn:primitive} shows that the Mori cone $\overline{\NE}(\PP_X(\E))$ is generated by the classes of finitely many rational curves $C_1,\cdots,C_n$ on $\PP_X(\E)$.
By the proof of \cite[Lemma 1.5]{SzWi:90}, we have $p^*(-K_X) \cdot C_i \geq 0$ for each $i$ since $\PP_X(\E)$ is weak Fano. 
Now for any integral curve $D$ in $X$, there exists an integral curve $C$ in $\PP_X(\E)$ such that $p(C) = D$.
Together with projection formula, we have $0 \leq p^*(-K_X) \cdot C = (-K_X) \cdot p_*C = m(-K_X \cdot D)$ for some positive integer $m \in \ZZ$ and thus it implies that $-K_X \cdot D \geq 0$. 
Therefore $X$ is also weak Fano.
\end{proof}

\subsection{Toric varieties of Picard rank $2$}\label{subsec:Kleinschmidt}

Kleinschmidt in \cite{Kle:88} studied smooth proper toric varieties of Picard rank $2$, and proved that they are all projective bundles of decomposable vector bundles over projective spaces. 
That is, if $X$ is such a variety of dimension $d'$, then there exists a positive integer $r < d'$ and a nondecreasing sequence of nonnegative integers $a_1, \cdots, a_{d'-r}$ such that
\[
	X \cong X_r(a_1, \cdots, a_{d'-r}) \colonequals \PP_{\PP^r}(\cO \oplus \cO(a_1) \oplus \cdots \oplus \cO(a_{d' - r})).
\] 
We call varieties constructed in this way \textit{Kleinschmidt varieties}.

Furthermore, these $r$, $a_1, \cdots, a_{d'-r}$ are uniquely determined, except for the isomorphism 
\[X_r(\underbrace{0, \cdots, 0}_{\text{$(d'-r)$ copies}}) \cong X_{d'-r}(\underbrace{0, \cdots, 0}_{\text{$r$ copies}})\] which corresponds to the isomorphism switching the factors in a direct product of projective spaces: $\PP^r \times \PP^{d'-r} \cong \PP^{d'-r} \times \PP^r$.

The projective bundle $X$ has $d' + 2$ rays which can be decomposed into a disjoint union of two primitive collections $\rays(X) = \X \sqcup \Y$ where
\[\X = \{x_{d' - r + 1},x_1,\cdots,x_{d' - r}\},\quad \Y = \{y_1, \cdots, y_{r + 1}\},\]
and the primitive generators of the rays in $\rays(X)$ are explicitly given by
\[\begin{array}{lll}
    x_{d'-r+1} & = (-1, \cdots, -1, & 0, \cdots, 0)   \\
    x_1 & = (1, \cdots, 0, & 0, \cdots, 0) \\
    & \vdots & \vdots  \\
    x_{d' - r} & = (0,\cdots, 1, & 0, \cdots, 0) \\
    y_1 & = (0, \cdots, 0, & 1, \cdots, 0) \\
    & \vdots & \vdots  \\
    y_{r} & = (0,\cdots, 0, & 0,\cdots, 1) \\
    y_{r + 1} & = (a_1,\cdots, a_{d'-r}, & -1, \cdots, -1).
\end{array}\]
The primitive relations corresponding to primitive collections $\X$ and $\Y$ are given respectively by
\begin{align}
    & \sum_{i = 1}^{d' - r + 1} x_i = 0, \label{eq:2primitivereln1} \\
    & \sum_{i = 1}^{r + 1}y_i = \sum_{i = 1}^{d' - r}a_ix_i. \label{eq:2primitivereln2}
\end{align}

Thus $X$ is Fano (\resp weak Fano) if $r + 1 - \sum_{i = 1}^{d' - r}a_i > 0$ (\resp $\geq 0$). We give complete lists of weak Fano toric threefolds and fourfolds of Picard rank $2$ in Appendices \ref{app:2Picard3Fold} and \ref{app:2Picard4Fold} respectively. 

The projective bundle $X$ has maximal cones given by $\Cone(\rays(X) \smallsetminus \{x_i,y_j\})$ where $x_i \in \X$ and $y_j \in \Y$. 
Thus any pair $\{D_{x_i},D_{y_j}\}$ forms a basis for $\Pic(X)$. To be consistent with the construction in Section \ref{subsec:projectiveBundle}, we use the basis $\{D_{x_{d' - r + 1}}, D_{y_{r + 1}}\}$ to describe elements of $\Pic(X)$ going forward.

\begin{proposition}\label{prop:positivity}
    The Cartier divisor $D = bD_{x_{d'-r+1}} + cD_{y_{r + 1}}$ on $X$ is ample (\resp nef) if $b > 0$ (\resp $b \geq 0$) and $c > 0$ (\resp $c \geq 0$). 
\end{proposition}
\begin{proof}
    Let $D = bD_{x_{d' - r + 1}} + cD_{y_{r + 1}}$. 
    By \cite[Theorem 6.4.9]{CLS:11}, $D$ is ample (\resp nef) if and only if $b > 0$ (\resp $b \geq 0$) and $c > 0$ (\resp $c \geq 0$) according to the primitive relations \eqref{eq:2primitivereln1} and \eqref{eq:2primitivereln2}.
\end{proof}

\subsection{Projective bundles of Picard rank $3$}\label{subsec:projectiveBundle}

By Proposition~\ref{prop:baseWeakFano}, any weak Fano toric variety of Picard rank $3$ with three primitive collections is the projectivization of a decomposable vector bundle over a weak Fano toric variety of Picard rank $2$. 
We can describe all such varieties of dimension $d$ as follows.

Let $r$, $d'$ be positive integers with $r < d' < d$ and let $a_1, \cdots, a_{d'-r}$ be a nondecreasing sequence of nonegative integers. 
Let $X$ be the projective bundle over $Y = \PP^{r}$ defined as in Section \ref{subsec:Kleinschmidt} by \[X = \PP_Y(\cO \oplus \cO(a_1) \oplus \cdots \oplus \cO(a_{d'-r})).\] 
Then given integers $b_i$ and $c_i$ for $1 \leq i \leq d-d'$, let $Z$ be the $\PP^{d - d'}$-bundle over $X$ defined by \[Z = \PP_{X}(\cO \oplus \cO(b_1,c_1) \oplus \cdots \oplus \cO(b_{d-d'},c_{d-d'}))\] 
where we assume all $b_i \geq 0$ and the pairs $\{(b_i,c_i)\}$ are ordered lexicographically. 
Here by $\mathcal{O}(b, c)$ we mean the line bundle associated to the divisor $bD_{x_{d'-r+1}}+cD_{y_{r+1}}$, using the basis for $\Pic(X)$ described in Section~\ref{subsec:Kleinschmidt}.
By \cite{Bat:91}, the toric variety $Z$ has three disjoint primitive collections 
\[\rays(Z) = \X \sqcup \Y \sqcup \Z = \{x_{d'-r+1},x_1,\cdots,x_{d'-r}\} \sqcup \{y_1,\cdots,y_{r + 1}\} \sqcup \{z_1,\cdots,z_{d-d'+1}\},\]
with the primitive relation $\sum_{i = 1}^{d - d' + 1} z_i = 0$ endowing $Z$ with the $\PP^{d-d'}$-bundle structure. 
The maximal cones of $Z$ are also easy to describe.
\begin{proposition}\label{prop:3-maxcones}
    Let $\rays(Z) = \X \sqcup \Y \sqcup \Z$ be a disjoint union of three primitive collections each of which consists of at least two rays. 
    Then the maximal cones of the toric variety $Z$ are all sets of type $\Cone(\rays(Z) \smallsetminus \{x_i,y_j,z_k\})$ where $x_i \in \X$, $y_j \in \Y$ and $z_k \in \Z$.
\end{proposition}

We give the other two primitive relations in the following proposition.

\begin{proposition}\label{prop:3-primitiverelns}
    Assume $b_i \geq 0$ for all $1 \leq i \leq d -d'$. Besides $\sum_{i = 1}^{d - d' + 1}z_i = 0$, two other primitive relations of $Z$ are expressed as follows.
    \begin{enumerate}[label = (\alph*)]
        \item If $c_i \geq 0$ for all $1 \leq i\leq d-d'$, then we have
        \begin{equation}\label{eq:primitivereln1}
            \sum_{i = 1}^{d' - r + 1} x_i = \sum_{i = 1}^{d - d'}b_iz_i,
        \end{equation}
        and 
        \begin{equation}\label{eq:primitivereln2}
            \sum_{i = 1}^{r + 1} y_i = \sum_{i = 1}^{d' - r} a_ix_i + \sum_{i = 1}^{d - d'}c_iz_i.
        \end{equation}
        
        \item If there exists an index $1 \leq j \leq d-d'$ such that $c_j < 0$, then instead of the relation \eqref{eq:primitivereln2}, we have the following relation: letting $l$ be the index such that $c_l = \min\{c_i \mid 1 \leq i \leq d - d'\} < 0$,
        \begin{equation}\label{eq:primitivereln4}
            \sum_{i = 1}^{r + 1}y_i = \sum_{i = 1}^{d' - r}a_ix_i + (-c_l)z_{d-d'+1} + \sum_{i = 1}^{d - d'}(c_i - c_l)z_i.
        \end{equation}
    \end{enumerate}
\end{proposition}
\begin{proof}
    Since the rays of $Z$ are given by
    \[\begin{array}{llll}
    x_{d' - r+1} & = (-1, \cdots, -1, & 0, \cdots, 0, & b_1,\cdots , b_{d-d'})   \\
    x_1 & = (1, \cdots, 0, & 0, \cdots, 0, & 0,\cdots, 0) \\
    & \vdots & \vdots & \vdots \\
    x_{d' - r} & = (0,\cdots, 1, & 0, \cdots, 0, & 0, \cdots, 0) \\
    y_1 & = (0, \cdots, 0, & 1, \cdots, 0, & 0, \cdots, 0) \\
    & \vdots & \vdots & \vdots \\
    y_{r} & = (0,\cdots, 0, & 0,\cdots, 1, & 0, \cdots, 0) \\
    y_{r + 1} & = (a_1,\cdots, a_{d'-r}, & -1, \cdots, -1, & c_1,\cdots, c_{d-d'}) \\
    z_1 & = (0,\cdots, 0, & 0, \cdots, 0, & 1, \cdots, 0) \\
    & \vdots & \vdots & \vdots \\
    z_{d-d'} & = (0,\cdots, 0, & 0, \cdots, 0, & 0, \cdots, 1) \\
    z_{d-d'+1} & = (0,\cdots, 0, & 0, \cdots, 0, & -1, \cdots, -1), \\
\end{array}\]
the result easily follows based on the description of the maximal cones (see Proposition \ref{prop:3-maxcones}).
\end{proof}

A basis for the Picard group $\Pic(Z)$ can be easily derived by the description of the rays listed in the proof of Proposition~\ref{prop:3-primitiverelns}.

\begin{corollary}\label{cor:3-picardbasis}
    The torus-invariant divisors $D_{x_{d' - r + 1}}$, $D_{y_{r + 1}}$ and $D_{z_{d - d' + 1}}$ generate the Picard group $\Pic(Z)$.
\end{corollary}

Using Proposition~\ref{prop:3-primitiverelns} and taking degrees of primitive relations give upper bounds on the values of the $b_i$'s and $c_i$'s for which the projective bundle $Z$ is Fano or weak Fano. 
We can also prove bounds on the $b_i$'s and $c_i$'s in another way, as follows.
\begin{theorem}
    Let $X$ and $\E = \cO \oplus \cO(b_1,c_1) \oplus \cdots \oplus \cO(b_{d-d'},c_{d-d'})$ be defined as above with $b_i \geq 0$ for all $i$. 
    If the projective bundle $Z = \PP_X(\E)$ is Fano (\resp weak Fano), then
    \begin{enumerate}[label = (\alph*)]
        \item If all  $c_i \geq 0$ for $1 \leq i \leq d-d'$, then we have inequalities
        \begin{align}
            & (d' - r + 1) - \sum_{i = 1}^{d-d'}b_i > 0 (\text{\resp } \geq 0), \label{eq:3bound1} \\
            & (r + 1) - \sum_{i = 1}^{d' - r}a_i - \sum_{i = 1}^{d - d'}c_i > 0 (\text{\resp } \geq 0); \label{eq:3bound2}
        \end{align}
        
        \item If there exists $c_j < 0$, let $l$ be the index such that $c_l = \min\{c_i \mid 1 \leq i \leq d - d'\}$, replace the inequality \eqref{eq:3bound2} by
        \begin{equation*}
            (r + 1) - \sum_{i = 1}^{d' - r}a_i + c_l - \sum_{i = 1}^{d - d'}(c_i - c_l) > 0 (\text{\resp } \geq 0).
        \end{equation*}
    \end{enumerate}
\end{theorem}

\begin{proof}
    Let $p : Z = \PP_X(\E) \to X$ be the projection. 
    Taking determinant of the relative Euler sequence,
    \[0 \to \Omega^1_{Z/X} \to p^*\E(-1) \to \cO_{Z} \to 0\]
    we have
    \[\det \Omega_{Z/X}^1 = p^*\det(\E) \otimes \cO_Z(-d + d' - 1).\]
    Thus the cotangent bundle sequence gives
    \[K_Z = p^*\det\Omega_X^1 \otimes \det \Omega^1_{Z/X} = p^*(K_X \otimes \det \E) \otimes \cO_Z(-d + d' - 1).\]
    Note that every torus-invariant curve on a proper toric variety is isomorphic to $\PP^1$. Let $C$ be a torus-invariant curve on $Z$, then by projection formula
    \[-K_Z \cdot C = -K_X \cdot p_*C + \det \E^\vee \cdot p_*C + \cO_Z(d - d' + 1) \cdot C.\]
    We have the following cases to consider
    \begin{enumerate}[label = (\roman*)]
        \item If $p_*C$ is a point in $X$ then $-K_Y \cdot C = d - d' + 1 > 0$ since $C$ is contained in a fiber of $p$.
        \item If $p_*C$ is a line in $X$, note that
        \begin{align*}
            & -K_X = D_{x_{d'-r+1}} + D_{x_1} + \cdots + D_{x_{d'-r}} + D_{y_1} + \cdots + D_{y_{r + 1}}, \\
            & \E = \cO \oplus \cO(b_1D_{x_{d'-r+1}} + c_1D_{y_{r+1}}) \oplus \cdots \oplus \cO(b_{d-d'}D_{x_{d'-r+1}} + c_{d-d'}D_{y_{r+1}}),
        \end{align*}
        then if $p_*C$ is the torus-invariant curve defined by the primitive relation \eqref{eq:2primitivereln1}, we have
        \[-K_X \cdot p_*C + \det \E^\vee \cdot p_*C = (d' - r + 1) - \sum_{i = 1}^{d - d'}b_i;\]
        if $p_*C$ is defined by the primitive relation \eqref{eq:2primitivereln2}, we have
        \[-K_X \cdot p_*C + \det \E^\vee \cdot p_*C = (r + 1) - \sum_{i = 1}^{d' - r}a_i - \sum_{i = 1}^{d - d'}c_i.\]
        Moreover, the intersection $\cO_Z(d - d' + 1) \cdot C$ can be negative, for instance when $C$ is contained in the Cartier divisor determined by the line bundle $\cO_Z(d - d' + 1)$. Now assume there is a $c_j$ such that $c_j < 0$, and let $c_l = \min\{c_i \mid 1 \leq i \leq d -d'\} < 0$. Consider the line bundle $\L = \cO_X(0, c_l)$ on $X$ and denote by
        \[\E' = \cO_X(0,-c_l) \oplus \cO_X(b_1,c_1 - c_l) \oplus \cdots \oplus \cO_X(b_{d - d'},c_{d-d'} - c_l).\]
        Since 
        \[\cO_{\PP \E} (1) = \cO_{\PP(\E' \otimes \L)}(1) = \cO_{\PP\E'}(1) \otimes p^*\L\]
        and $\cO_{\PP\E'}(1)$ is nef on $Z$ by Proposition \ref{prop:positivity} as $b_i \geq 0$ and $c_i - c_l \geq 0$ for all $1 \leq i \leq d - d'$, $\cO_{\PP\E'}(1)\cdot C \geq 0$. 
        Moreover, by projective formula we have $p^*\L \cdot C = (c_lD_{y_{r + 1}}) \cdot p_*C$, so $p^*\L \cdot C = 0$ if $p_*C$ is given by the primitive relation~\eqref{eq:2primitivereln1}, and $p^*\L \cdot C = c_l$ if $p_*C$ is given by relation~\eqref{eq:2primitivereln2}.
    \end{enumerate}
    This proves the proposition.
\end{proof}

In the code repository associated to this paper \cite{weakFano-repo}, we construct all possible input triples $(d, a, l)$ where $d$ is the dimension of the variety, $a$ is the list $a_1, \cdots, a_{d'-r}$, and $l$ is a list of all pairs $\{b_1,c_1\}, \cdots, \{b_{d-d'}, c_{d-d'}\}$ such that the associated variety is weak Fano. 
The function \texttt{projectiveBundleConstructor} takes such a triple as input and constructs the associated variety as a \texttt{NormalToricVariety} object in Macaulay2. 

Since the same variety can be constructed in multiple ways as a projective bundle, we also write a function called \texttt{areIsomorphic} which takes as input two \texttt{NormalToricVariety} objects and returns true if and only if they are isomorphic. 
Recall that a toric morphism (\cite[Theorem 3.3.4]{CLS:11}) between two normal toric varieties $X(\Sigma)$ and $X(\Sigma')$ can be given by a homomorphism of lattices that maps every cone in $\Sigma$ into a cone in $\Sigma'$. 
A toric morphism is an isomorphism if it is an isomorphism of abelian groups and maps cones onto cones isomorphically, defining a bijection between two sets of cones. 
The Macaulay2 function \texttt{areIsomorphic} checks if a toric isomorphism can exist between two \texttt{NormalToricVariety} objects by checking if there is an invertible linear map defined over $\ZZ$ that maps rays to rays and cones to cones bijectively. Since (projective) toric varieties are isomorphic if and only if there is a \textit{toric} isomorphism between them \cite[Theorem 4.1]{Ber:03}, \texttt{areIsomorphic} returns true precisely when two projective toric varieties are isomorphic.

By constructing all possible projective bundles (in a fixed dimension) and removing redundant isomorphic copies from the list, we obtain a full classification of weak Fano toric varieties of Picard rank $3$ with three primitive collections.

\begin{theorem}
    There are $18$ isomorphism classes of rank $3$ weak Fano toric threefolds with three primitive collections, as described in Appendix~\ref{app:3Picard3Fold}. 
    There are $81$ isomorphism classes of rank $3$ weak Fano toric fourfolds with three primitive collections, as described in Appendix~\ref{app:3Picard4Fold}.
\end{theorem}
\begin{remark}
The Chern numbers $c_1^4$ and $c_1^2c_2$ of toric Fano fourfold D8 given in Batyrev's list of Fano fourfolds \cite{Bat:99} are incorrectly computed. 
See entry 4.PB-15 in Appendix~\ref{app:3Picard4Fold} for the correct invariants. 
\end{remark}

\begin{remark}
	Sato in \cite{Sat:02} classifies \textit{weakened} toric threefolds, that is, weak Fano toric threefolds that are not Fano but can be deformed to Fano threefolds under a small deformation. 
    There are two weakened toric threefolds of Picard rank $3$: $\PP^1 \times \FF_2$ and a variety denoted by $X_3^0$, which has anticanonical degree $52$.
    Since Chern numbers are deformation invariants, we see the latter variety is isomorphic to 3.PB-17 in Appendix~\ref{app:3Picard3Fold}. 
    $\PP^1 \times \FF_2$ deforms to the Fano threefold $\PP^1 \times \PP^1 \times \PP^1$ and $X_3^0$ deforms to Fano threefold 3-31.
    
    Sato in \cite{Sat:21} studies \textit{special weak} toric varieties and classifies them in dimension four (in dimension three, special weak toric varieties are exactly the weakened toric varieties). 
    Two of the varieties in the list, denoted by $\Z_1$ and $\Z_{14}$, have Picard rank $3$. By inspection of the primitive relations, we can see that $\Z_1$ is isomorphic to 4.PB-42 in Appendix~\ref{app:3Picard4Fold} and $\Z_{14}$ is isomorphic to 4.PB-60. 
\end{remark}

\section{Batyrev's construction}\label{sec:batyrevconstruction}

Let's briefly recall Batyrev's construction of smooth proper toric varieties of Picard rank $3$ with five primitive collections \cite[Theorem 6.6]{Bat:91}.

Fix the dimension $d$. Let $(p_0,p_1,p_2,p_3,p_4)$ be a positive partition of $d + 3$, which is the number of rays. 
Furthermore, choose nonnegative integers $b_1, \cdots, b_{p_3}$ and $c_2, \cdots, c_{p_2}$ (note that the indices for $c_i$'s start from $2$). 
Batyrev in \cite{Bat:91} proves that we can associate to this data a toric variety as follows. 
For each $i \in \ZZ/5\ZZ$ we can find a collection of vectors $X_i$ with $|X_i| = p_i$, such that these vectors are primitive generators of the rays of a regular and complete fan.
The five sets $P_i = X_i \cup X_{i+1}$ are the primitive collections of the fan, and we have the primitive relations
\begin{equation}\label{eq:5primitivereln}
\begin{split}
    & (v_1 + \cdots + v_{p_0}) + (y_1 + \cdots + y_{p_1}) = (c_2z_2 + \cdots + c_{p_2}z_{p_2}) + ((b_1 + 1)t_1 + \cdots + (b_{p_3} + 1)t_{p_3}), \\
    & (y_1 + \cdots + y_{p_1}) + (z_1 + \cdots + z_{p_2}) = (u_1 + \cdots + u_{p_4}), \\
    & (z_1 + \cdots + z_{p_2}) + (t_1 + \cdots + t_{p_3}) = 0, \\
    & (t_1 + \cdots + t_{p_3}) + (u_1 + \cdots + u_{p_4}) = (y_1 + \cdots + y_{p_1}), \\
    & (u_1 + \cdots + u_{p_4}) + (v_1 + \cdots + v_{p_0}) = (c_2z_2 + \cdots + c_{p_2}z_{p_2}) + (b_1t_1 + \cdots + b_{p_3}t_{p_3}),
\end{split}
\end{equation}
where \[X_0 = \{v_i\}_{i = 1}^{p_0},\quad X_1 = \{y_i\}_{i = 1}^{p_1},\quad X_2 = \{z_i\}_{i = 1}^{p_2},\quad X_3 = \{t_i\}_{i = 1}^{p_3},\quad X_4 = \{u_i\}_{i = 1}^{p_4}.\]
Then the toric variety associated to this fan is smooth and proper of Picard rank $3$ with five primitive collections, and is uniquely determined by the $p_i$, $b_i$ and $c_i$ up to isomorphism \cite[Theorem 6.6]{Bat:91}. 
Let's denote this variety, which is well-defined up to isomorphism, as 
\[\Bat(\{p_i\}, \{b_i\}, \{c_i\}) \coloneqq \Bat(\{p_0, p_1, p_2, p_3, p_4\}, \{b_1, \cdots, b_{p_3}\}, \{c_2, \cdots, c_{p_2}\}).\]
We will call varieties that can be constructed in this way \textit{Batyrev varieties}. More explicitly, one can construct $\Bat(\{p_i\}, \{b_i\}, \{c_i\})$ by labelling the standard basis of $\ZZ^d$ by 
\[ \{ v_1, \dots v_{p_0}, y_2, \dots y_{p_1}, z_2, \dots z_{p_2}, t_1, \dots t_{p_3}, u_2, \dots u_{p_4} \}\]
and then setting
\begin{align*}
    & z_1 = -z_2 - \cdots -z_{p_2} -t_1 - \cdots - t_{p_3}, \\
    & y_1 = -y_2 - \dots - y_{p_1} - z_1 - \cdots - z_{p_2} + u_1 + \cdots + u_{p_4}, \\
    & u_1=-u_2- \cdots -u_{p_4}-v_1- \cdots -v_{p_0} + c_2z_2 + \cdots + c_{p_2}z_{p_2} + b_1t_1 + \cdots +b_{p_3}t_{p_3}.
\end{align*}
Then these vectors satisfy all the primitive relations and thus generate the rays that define the fan of the Batyrev variety.

Evidently, any rearrangement of the $b_i$'s and $c_i$'s produces the same toric variety up to isomorphism, so let's assume $b_1, \cdots, b_{p_3}$ and $c_2, \cdots, c_{p_2}$ are \textit{nondecreasing} sequences.

It is stated in \cite[Theorem 6.6]{Bat:91} that the data $(\{p_i\}, \{b_i\}, \{c_i\})$ is unique for a given isomorphism type of a Batyrev variety, up to a ``symmetry of the pentagon''. 
To explain this remark, consider a pentagon whose vertices correspond to the five sets $X_0, \cdots, X_4$, so that the edges correspond to the primitive collections $P_0, \cdots, P_4$. Let $D_5$ be the dihedral group of order $10$, \ie the group of symmetries of the pentagon, so it acts on both the sets $\{X_0, \cdots, X_4\}$ and $\{P_0, \cdots, P_4\}$.  


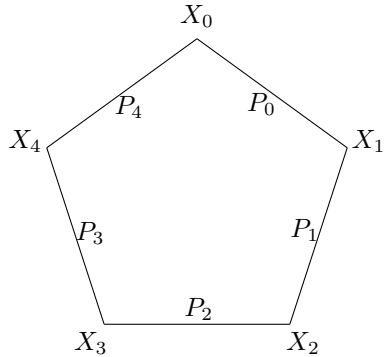
\begin{center}
\begin{tikzpicture}[scale=3, every node/.style={font=\small}]
  \def\radius{0.7}
  \def\labelradius{0.8}
  \def\labeledge{0.5}

  \foreach \i in {0,...,4} {
    \coordinate (X\i) at ({90 - \i*72}:\radius); 
    \node at ({90 - \i*72}:\labelradius) {$X_{\i}$}; 
  }

  \foreach \i in {0,...,4} {
    \pgfmathtruncatemacro{\j}{mod(\i+1,5)}
	\draw (X\i) -- (X\j) node[] {};
        \node at ({55 - \i*72}:\labeledge) {$P_{\i}$};
    }
        
 
\end{tikzpicture}
\end{center}

Furthermore, let $\Bat(\{p_i\}, \{b_i\}, \{c_i\})$ and $ \Bat(\{p_i'\}, \{b_i'\}, \{c_i'\})$ be two Batyrev varieties and suppose $\tau \in D_5$. Then if $p'_i = p_{\tau(i)}$ for all $i \in \ZZ/5\ZZ$ we can define a map from the set the rays of the fan defining $\Bat(\{p_i\}, \{b_i\}, \{c_i\})$ to the rays of $ \Bat(\{p_i'\}, \{b_i'\}, \{c_i'\})$, by sending the $j$th element of $X_{\tau(i)}$ to the $j$th element of $X'_{i}$. This defines an isomorphism of toric varieties (with the inverse coming from $\tau^{-1}$) if and only if $\tau$ preserves all primitive relations. 
We show this can only occur if all $b_i$'s and $c_i$'s (and $b'_i$'s and $c'_i$'s) are zero and $\tau$ is the reflection about $X_1$, \ie reflecting along the line connecting the center of the pentagon and the vertex denoted by $X_1$.

\begin{proposition}
Let $(\{p_i\}, \{b_i\}, \{c_i\})$, $(\{p_i'\}, \{b_i'\}, \{c_i'\})$ be two distinct triples of lists of nonnegative integers as above, so that they define Batyrev varieties.
Then there exists an isomorphism $\Bat(\{p_i\}, \{b_i\}, \{c_i\}) \cong \Bat(\{p_i'\}, \{b_i'\}, \{c_i'\})$ if and only if all $b_i$'s, $c_i$'s, $b'_i$'s, and $c'_i$'s are zero, $p_0 \neq p_2$ or $p_3 \neq p_4$, and $(p'_0, p'_1, p'_2, p'_3, p'_4) = (p_2,p_1,p_0,p_4,p_3)$.
\end{proposition}

\begin{proof}
First assume all $b_i$'s and $c_i$'s are zero. 
One can easily see that the symmetry of the pentagon with vertices $X_0, \cdots, X_4$ which reflects about $X_1$, thus interchanging $X_0$ with $X_2$ and $X_3$ with $X_4$, preserves primitive relations and thus gives an isomorphism of toric varieties.
Conversely, assume we have an isomorphism $\Bat(\{p_i\}, \{b_i\}, \{c_i\}) \cong \Bat(\{p_i'\}, \{b_i'\}, \{c_i'\})$. 
There must be a nontrivial symmetry of the pentagon $\tau \in D_5$ such that $p_{\tau(i)} = p'_i$ that preserves primitive relations. 
If any $b_i$ is nonzero, then looking at the primitive relations~\eqref{eq:5primitivereln} we see the first and last primitive relations have on the right side the same vectors $t_i$ but the coefficients in the first primitive relation are $1$ greater than those in the last primitive relation. 
This characterizes the first and last primitive relations: thus $\tau$ must fix the first and last primitive collections,  but $\tau(P_0) = P_0$ and $\tau(P_4) = P_4$ is impossible if $\tau$ is nontrivial. 
We have a contradiction: thus all the $b_i$ are zero. Next, if any $c_i$ are nonzero, then we see that first and last primitive relations have $z_i$'s on the right, and the first also has $t_i$'s. Thus $P_0$ and $P_4$ must be fixed by $\tau$, which is impossible: thus all $c_i$ are zero. 
When all the $b_i$ and $c_i$ are zero, we see the only possible symmetry is one which interchanges $P_2$ and $P_4$, since these are the primitive collections with sum zero. 
Thus $\tau$ must be the element of $D_5$ which is reflection about $X_1$: this symmetry interchanges $X_0$ with $X_2$ and $X_3$ with $X_4$. 
So all $b_i'$'s and $c_i'$'s are also zero.
Since the triples $(\{p_i\}, \{b_i\}, \{c_i\})$, $(\{p_i'\}, \{b_i'\}, \{c_i'\})$ are distinct, either $p_0 \neq p_2$ or $p_3 \neq p_4$.
\end{proof}

Since we have the primitive relations described explicitly in terms of the input data, by computing the degrees of the primitive relations it is easy to check if a toric variety is (weak) Fano.

\begin{corollary}
    The Batyrev variety constructed above is Fano if and only if the following inequalities are satisfied. 
    \[\begin{cases}
    & p_1 + p_2 - p_4 > 0, \\
    & p_3 + p_4 - p_1 > 0, \\
    & p_0 + p_1 - p_3 > \sum_{i = 2}^{p_2} c_i + \sum_{i = 1}^{p_3} b_i, \\
    & p_0 + p_4 > \sum_{i = 2}^{p_2} c_i + \sum_{i = 1}^{p_3} b_i.
    \end{cases}\]
    It is weak Fano if and only if all strict inequalities $>$ are replaced by $\geq$.
\end{corollary} 

\begin{example}
For a Batyrev surface we must have $p_0 = \dots = p_4 = 1$, so there is only $b_1$ and no $c_i$'s. By the criterion above, the only Fano Batyrev surface is $\Bat(\{1,1,1,1,1\},\{0\},\{\})$: this is the surface $\text{Bl}_1(\PP^1 \times \PP^1)$ from Example \ref{ex:fanobatyrevsurface}. The only other weak Fano Batyrev surface is $\Bat(\{1,1,1,1,1\},\{1\},\{\})$. Since every smooth projective toric surface of Picard rank $3$ has five primitive collections, these are the only weak del Pezzo surfaces of Picard rank $3$: they are called $S_7$ and $W_3$ respectively in Sato's notation \cite{Sat:02}. 
\end{example}

In the code repository \cite{weakFano-repo} we have a function that generates all triples $\{p, b, c\}$ that satisfy these inequalities: there are finitely many in each dimension. 
The authors, along with Thiago Holleben, wrote the function \texttt{batyrevConstructor} at the 2024 Macaulay2 workshop, which takes as input the list $\{p, b, c\}$ and constructs the associated variety as a \texttt{NormalToricVariety} object. 
It implements the function $\Bat$ described above. 
In order to avoid redundant copies we require that if all $b_i$'s and $c_i$'s are zero, the pairs $(p_0, p_3)$ and $(p_2, p_4)$ must be in lexicographic order. 
Using this, we complete the classification of ``Batyrev varieties'' in dimensions three and four. 
We also have a function that automates the entire process, which automatically generates all Batyrev varieties in any dimension. 

\begin{theorem}
    There are $10$ isomorphism classes of rank $3$ weak Fano toric threefolds with five primitive collections, as described in Appendix~\ref{app:3Picard3Fold}. 
    There are $33$ isomorphism classes of rank $3$ weak Fano toric fourfolds with five primitive collections, as described in Appendix~\ref{app:3Picard4Fold}.
\end{theorem}

\section{Star subdivisions and blowdowns}

In this section, we discuss star subdivisions, blowdowns and flops of smooth projective toric varieties of Picard rank $3$. Star subdivisions and blowdowns of toric varieties with complete simplicial fans are studied in \cite[Section 4]{Sat:00} in terms of primitive collections and primitive relations. 
We write Macaulay2 functions for implementing these functions: specifically, \cite[Theorem 4.3]{Sat:00} and \cite[Corollary 4.9]{Sat:00} are implemented by Macaulay2 functions \texttt{primitiveBlowup} and \texttt{primitiveBlowdown} respectively in \cite{weakFano-repo}.

By studying blowdowns we can understand how the rank $3$ weak Fano toric varieties we have constructed can contract to simpler varieties of smaller Picard rank such as Kleinschmidt varieties. Also, we can computationally investigate when various rank $3$ weak Fano toric varieties flop to each other, including the interesting phenomenon of a variety with three primitive collections flopping to a variety with five primitive collections.

Let's first recall the definition of star subdivision of a complete simplicial fan.

\begin{definition}[Star subdivision]
    Let $\Sigma$ be a complete simplicial fan in $N_\RR$ and $\sigma = \Cone(u_1,\cdots, u_l)$ a cone in $\Sigma$. 
    Let $u \in N$ be a primitive vector contained in the relative interior of $\sigma$.
    The star subdivision of $\Sigma$ relative to $(\sigma,u)$ is the fan (again complete and simplicial) defined by
    \[\Sigma^*(\sigma,u) = \{\tau \in \Sigma \mid \sigma \not\subseteq \tau\} \cup \bigcup_{\sigma\subseteq\tau\in \Sigma} \{\Cone(R) \mid R \subseteq \{u\} \cup \tau(1), \sigma(1) \not\subseteq R\}.\]
    In particular, if $\sigma$ is a smooth cone with $u = \sum_{i = 1}^l u_i$, then the star subdivision is
    \[\Sigma^*(\sigma,u) = (\Sigma \smallsetminus \{\sigma\}) \cup \Sigma'(\sigma)\]
    where $\Sigma'(\sigma)$ is the set of all cones generated by subsets of $\{u,u_1,\cdots,u_l\}$ except $\{u_1,\cdots,u_l\}$.
    In this case we write $\Sigma^*(\sigma)$ instead of $\Sigma^*(\sigma,u)$.
\end{definition}

If all cones of $\Sigma$ which contain $\sigma$ are smooth and $u = \sum_{i = 1}^l u_i$, then the refinement $\Sigma^*(\sigma)$ of $\Sigma$ induces a toric morphism $X(\Sigma^*(\sigma)) \to X(\Sigma)$ which is a (equivariant) blowup of $X(\Sigma)$ along the orbit closure with respect to $\sigma$. 

The inverse transformation of a star subdivision is called a (equivariant) blowdown along $(P,u)$ where $P = \rays(\sigma)$ is a primitive collection of $\Sigma^*(\sigma,u)$. 
We call a star subdivision (\resp blowdown) \textit{regular} if the fans before and after the star subdivision (\resp blowdown) are regular. 
Hence a necessary condition for the existence of a regular blowdown is the existence of a primitive collection $P = \{u_1,\cdots, u_l\}$ with primitive relation $\sum_{i = 1}^l u_i = u$. 
In this case, we can simply refer to the blowdown along $P$ instead of $(P,u)$.

Kleinschmidt proves that any projective bundle $X = \PP_{\PP^r}(\cO^{\oplus d'-r-1} \oplus \cO(1))$ can be further blown down to $\PP^{d'}$ \cite[Theorem 1]{Kle:88}. 

\begin{proposition}\label{prop:2-blowndown}
    Let $X = \PP_{\PP^r}(\cO^{\oplus d'-r} \oplus \cO(a))$, $a > 0$. Then $X$ can be further blown down to a weighted projective space of dimension $d'$:
    \[\PP(\underbrace{1,\cdots, 1}_{\text{$(r + 1)$ copies}},\underbrace{a,\cdots, a}_{\text{$(d'-r)$ copies}}).\]
\end{proposition}
\begin{proof}
    The Kleinschmidt variety $X$ has disjoint primitive collections $\X = \{x_1,\cdots, x_{d'-r},x_{d'-r+1}\}$ and $\Y = \{y_1,\cdots, y_{r+1}\}$ which satisfy the primitive relations
    \[\sum_{i = 1}^{d'-r+1} x_i = 0,\quad \sum_{i = 1}^{r + 1} y_i = ax_{d' - r}.\]
    By \cite[Corollary 4.9]{Sat:00}, after the blowdown along $(\Y, x_{d'-r})$, the primitive collection is a single set of rays $(\X \cup \Y)\smallsetminus \{x_{d'-r}\}$ satisfying the relation
    \[ax_{d'-r+1} + a\sum_{i = 1}^{d'-r-1}x_i + \sum_{i = 1}^{r + 1}y_i = 0,\]
    and thus the result follows.
\end{proof}

By \cite[Example 8.3.3]{CLS:11}, a weighted projective space $\PP(q_0,\cdots,q_n)$ with reduced weight vector $(q_0,\cdots,q_n)$ is Gorenstein Fano if and only if each $q_j$ divides $\sum q_i$. 
Thus the weighted projective space above is Gorenstein Fano if and only if $a$ divides $r+1$.

Starting from now on,
we will only discuss regular star subdivisions and blowdowns for toric varieties of Picard rank $3$.
In \cite[Theorem 4.10]{Sat:00}, Sato provides a criterion for the existence of a regular blowdown.
For toric varieties of Picard rank $3$, adapting the theorem to the case of smooth toric varieties of five primitive collections implies the following corollary.

\begin{corollary}\label{cor:blowdown5PC}
    Let $X(\Sigma^*)$ be a smooth projective toric variety of Picard rank $3$ with five primitive collections, \ie a Batyrev variety, constructed in Section~\ref{sec:batyrevconstruction}. Then $X(\Sigma^*)$ can only have the following four types of equivariant blowdown to a smooth projective toric variety of Picard rank $2$.
    \begin{enumerate}[label = (\alph*)]
        \item When $p_3 = 1$, with $b_1 = 0$ and all $c_i$'s vanishing, along primitive collection $(P_0,t_1)$.
        \item When $p_4 = 1$, along primitive collection $(P_1,u_1)$;
        \item When $p_1 = 1$, along primitive collection $(P_3,y_1)$;
        \item When $p_2 = 1$, with $c_1 = 1$ and all $b_i$'s vanishing, along primitive collection $(P_4,z_1)$.
    \end{enumerate}
\end{corollary}

Moreover, if the variety $X(\Sigma^*)$ we start with is Fano, then the regular blowdowns in Corollary \ref{cor:blowdown5PC} remain Fano.
In general, if $X(\Sigma^*)$ is weak Fano, the blowdowns in Corollary \ref{cor:blowdown5PC} are always weak Fano, and are Fano depending on the following degree conditions. 
Let $\Sigma$ be the fan after blowdown and $X(\Sigma^*)$ is weak Fano,
\begin{enumerate}[label = (\alph*)]
    \item If $\Sigma$ is obtained by blowdown along $(P_0,t_1)$, then $X(\Sigma)$ is Fano if $\deg P_1 > 0$; 
    \item If $\Sigma$ is obtained by blowdown along $(P_1,u_1)$, then $X(\Sigma)$ is Fano if $\deg P_0 > 0$;
    \item If $\Sigma$ is obtained by blowdown along $(P_3,y_1)$, then $X(\Sigma)$ is Fano if $\deg P_4 > 0$;
    \item If $\Sigma$ is obtained by blowdown along $(P_4,z_1)$, then $X(\Sigma)$ is Fano if $\deg P_3 > 0$.
\end{enumerate}

\begin{remark}
    It can be easily seen from Corollary~\ref{cor:blowdown5PC} that any weak Fano toric threefold with five primitive collections admits a regular blowdown to a weak Fano toric variety of Picard rank $2$, \ie a Kleinschmidt variety.
    This is not true in higher dimensions: for example, 4.BC-4 and 4.BC-16 do not have such a regular blowdown.
\end{remark}

An equivariant blowdown to a smooth toric variety may not exist even for weak Fano toric threefolds with three primitive collections, for instance 3.PB-8, 3.PB-9, 3.PB-10 and 3.PB-18.
Moreover, even if such blowdown exists, the blowdown not necessarily remains weak Fano.
Consider the construction of a toric variety with three primitive collections discussed in Section~\ref{subsec:projectiveBundle}, and let
\[Z = \PP_X(\cO \oplus \cO(b_1,c_1) \oplus \cdots \oplus \cO(b_{d - d'},c_{d - d'}))\]
where $X = \PP_Y(\cO \oplus \cO(a_1) \oplus \cdots \oplus \cO(a_{d'-r}))$ is a Kleinschmidt variety over $Y = \PP^r$.
Note that both $\{a_1,\cdots, a_{d'-r}\}$ and $\{b_1,\cdots, b_{d - d'}\}$ are defined to be finite sequences of weakly increasing nonnegative integers, but $\{c_1,\cdots, c_{d -d'}\}$ may not be.
Recall that $Z$ has three disjoint primitive collections
\[\X = \{x_{d'-r+1},x_1,\cdots,x_{d'-r}\},\quad \Y = \{y_1,\cdots,y_{r + 1}\},\quad \Z = \{z_1,\cdots, z_{d-d'+1}\}.\]

Adapting \cite[Theorem 4.10]{Sat:00} to the case of smooth toric varieties of three primitive collections implies the following corollary.

\begin{corollary}\label{cor:blowdown3PC}
    Let $Z$ be a weak Fano toric variety with three primitive collections defined as above.
    Depending on the values of $c_i$'s, the toric variety $Z$ can only have the following types of equivariant blowdown to smooth projective toric varieties of Picard rank $2$.
    \begin{itemize}
        \item[(i)] If $c_i \geq 0$ for all $1 \leq i \leq d - d'$, then
        \begin{itemize}
            \item[(a)] when $b_i = 0$ for all $1 \leq i \leq d - d' - 1$ but $b_{d - d'} = 1$, 
            blow down $Z$ along the primitive collection $(\X, z_{d - d'})$.
            The blowdown variety is a $\PP^{d - r}$-bundle over $\PP^r$ with two primitive collections given by $\Y$ and $(\Z \smallsetminus \{z_{d - d'}\}) \cup \X$. The blowdown toric variety remains weak Fano if $c_{d - d'} = 0$;
            otherwise, if $Z$ is not Fano, then the blowdown variety fails to be weak Fano when $c_{d - d'} > 0$ (\eg 4.PB-25).
            Moreover, even if $Z$ is Fano, its blowdown (if exists) may not remain weak Fano (\eg 4.PB-12);

            \item[(b)] when $c_i = 0$ for all $1 \leq i \leq d - d'$, and $a_i = 0$ for all $1 \leq i \leq d' - r - 1$ but $a_{d' - r} = 1$, 
            blow down $Z$ along the primitive collection $(\Y, x_{d' - r})$. 
            The blowdown variety is a $\PP^{d - d'}$-bundle over $\PP^{d'}$ with two primitive collections $\Z$ and $(\X \smallsetminus \{x_{d' - r}\}) \cup \Y$. 
            The blowdown toric variety is always Fano;

            \item[(c)] when $a_i = 0$ for all $1 \leq i \leq d' - r$, and $c_j = 1$ but $c_i = 0$ for all $i \neq j$,
            blow down $Z$ along the primitive collection $(\Y, z_j)$.
            The blowdown variety is a $\PP^{d - d' + r}$-bundle over $\PP^{d' - r}$ with two primitive collections $\X$ and $(\Z\smallsetminus \{z_j\}) \cup \Y$. 
            The blowdown variety remains weak Fano if $b_j = 0$;
            otherwise, if $Z$ is not Fano, then the blowdown is not weak Fano if $b_j > 0$ (\eg 4.PB-51).
            Moreover, even if $Z$ is Fano, its blowdown (if exists) may not remain weak Fano (\eg 4.PB-11).
        \end{itemize}

        \item[(ii)] If there exists $c_j < 0$, let $c_l = \min\{c_i \mid 1 \leq i \leq d - d'\}$, then
        \begin{itemize}
            \item[(a)] when $b_i = 0$ for all $1 \leq i \leq d - d' + 1$ but $b_{d - d'} = 1$,
            blow down $Z$ along the primitive collection $(\X, z_{d - d'})$.
            The blowdown variety is a $\PP^{d - r}$-bundle over $\PP^r$ with two primitive collections $\Y$ and $(\Z \smallsetminus \{z_{d - d'}\}) \cup \X$.
            The blowdown toric variety remains weak Fano if $c_{d - d'} = c_l$;
            otherwise, if $Z$ is not Fano, then the blowdown fails to be weak Fano if $c_{d - d'} > c_l$ (\eg 4.PB-23);
    
            \item[(b)] when $a_i = 0$ for all $1 \leq i \leq d' - r$, and $c_i = c_l = -1$ for all $1 \leq i \leq d - d'$,
            blow down $Z$ along the primitive collection $(\Y, z_{d - d' + 1})$.
            The blowdown variety is a $\PP^{d - d' + r}$-bundle over $\PP^{d' - r}$ with two primitive collections $\X$ and $(\Z \smallsetminus \{z_{d - d' + 1}\}) \cup \Y$.
            The blowdown toric variety remains weak Fano (\resp Fano) if $Z$ is weak Fano (\resp Fano).
        \end{itemize}
    \end{itemize}
\end{corollary}

In the code repository \cite{weakFano-repo} associated to this paper, we write functions \texttt{primitiveBlowup} and \texttt{primitiveBlowdown} which respectively implement \cite[Theorem 4.3]{Sat:00} for regular star subdivisions, and \cite[Corollary 4.9]{Sat:00} for regular blowdowns. 
More precisely, the function \texttt{primitiveBlowup} takes as input a pair $(X, l)$, where $X$ is a normal toric variety with a list $l$ representing a \textit{cone} $\sigma_l$ in the fan of $X$, and constructs the star subdivision of $X$ along the cone $\sigma_l$. 
Conversely, the function \texttt{primitiveBlowdown} takes as input a pair $(X, l)$, where $X$ is a normal toric variety with a list $l$ representing a \textit{primitive collection}, and constructs a smooth toric variety $X'$ such that $X \to X'$ is a star subdivision of $X'$ along $\sigma_l$, which is a cone in the fan of $X'$.

We now recall the definition of a \textit{flop} for a smooth projective toric variety.

\begin{definition}\cite[Definition 6.18]{Sat:00}
    Let $X$ be a smooth projective toric variety. 
    Let $P = \{x_1,\cdots, x_l\} \subseteq \rays(X)$ be a primitive collection with primitive relation $\sum_{i = 1}^l x_i = \sum_{i = 1}^l y_i$. Assume that such relation is contained in an extremal ray of the Mori cone $\overline{\NE}(X)$. Then $X$ admits a star subdivision along $\Cone(y_1,\cdots,y_l)$ followed by a blowdown along $P$ to another toric variety $X'$. We call this operation $X \to X'$ a flop.
\end{definition}

\begin{example}[4.PB-33]\label{eg:4.PB-33flop}
    Consider the toric fourfold $X$ (4.PB-33) from Appendix~\ref{app:3Picard4Fold} with three primitive collections, which is a $\PP^2$-bundle over the Hirzebruch surface $\FF_1$.
    Applying \texttt{primitiveBlowdown} to this toric variety along the primitive collection given by rays indexed by $\{0, 1\}$, we obtain a weak Fano toric variety of Picard rank $2$ which is isomorphic to $X_4(0,1,1)$.
    Moreover, notice that $X$ has a primitive relation $x_2 + x_3 = x_1 + x_4$ associated to the primitive collection indexed by $\{2, 3\}$.
    First applying \texttt{primitiveBlowup} to the toric variety $X$ along the cone given by the list $\{1, 4\}$, we obtain a toric variety $X^*$ which is a star subdivision of $X$. Then applying \texttt{primitiveBlowdown} to $X^*$ along a primitive collection given by the list $\{2, 3\}$, we obtain a new toric variety $X'$ with five primitive collections, which is isomorphic to 4.BC-19. 
    Thus 4.PB-33 flops to 4.BC-19.
    This is an example of a toric variety with three primitive collections which flops to a toric variety with five primitive collections.
\end{example}

\begin{example}[non-projective]\label{eg:nonprojectiveflop}
Consider the following primitive vectors 
\begin{align*}
& x_0 = (-1,-1,-1),\quad x_1 = (1,0,0),\quad x_2 = (0,1,0),\quad x_3 = (0,0,1),\\
& x_4 = (0,-1,-1),\quad x_5 = (-1,0,-1),\quad x_6 = (-1,-1,0).
\end{align*}
Let $\Sigma_1$ and $\Sigma_2$ be the fans whose cones are generated by the subsets of these vectors indicated below.

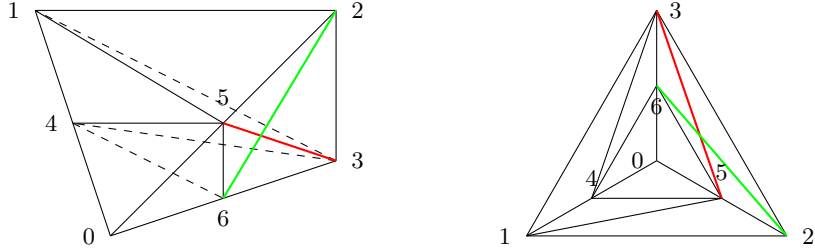
\[\begin{tikzpicture}[scale=1,>=stealth]
\draw[-] (-1,2) -- (3,2) node[]{};
\draw[-] (-1,2) -- (0,-1) node[]{};
\draw[-] (0,-1) -- (3,2) node[]{};
\draw[-] (3,0) -- (0,-1) node[]{};
\draw[-] (3,0) -- (3,2) node[]{};
\draw[dashed] (3,0) -- (-1,2) node[]{}; 
\draw[-] (-0.5,0.5) -- (1.5,0.5) node[]{};
\draw[-] (1.5,0.5) -- (1.5,-0.5) node[]{};
\draw[dashed] (-0.5,0.5) -- (1.5,-0.5) node[]{}; 
\draw[-] (-1,2) -- (1.5,0.5) node[]{};
\draw[dashed] (3,0) -- (-0.5,0.5) node[]{}; 
\draw[-,red, thick] (1.5,0.5) -- (3,0) node[]{};
\draw[-,green, thick] (3,2) -- (1.5,-0.5) node[]{};
\fill (3,2) circle (0pt) node[right=2pt] {\footnotesize{$2$}};
\fill (-1,2) circle (0pt) node[left=2pt] {\footnotesize{$1$}};
\fill (-0.5,0.5) circle (0pt) node[left=2pt] {\footnotesize{$4$}};
\fill (3,0) circle (0pt) node[right=2pt] {\footnotesize{$3$}};
\fill (0,-1) circle (0pt) node[left=2pt] {\footnotesize{$0$}};
\fill (1.5,-0.5) circle (0pt) node[below=1pt] {\footnotesize{$6$}};
\fill (1.5,0.5) circle (0pt) node[above=3pt] {\footnotesize{$5$}};
\end{tikzpicture}
\qquad\qquad
\begin{tikzpicture}[scale=1,>=stealth]
\draw[-] (-1.73,-1.5) -- (0,1.5) node[]{};
\draw[-] (1.73,-1.5) -- (0,1.5) node[]{};
\draw[-] (-1.73,-1.5) -- (1.73,-1.5) node[]{};
\draw[-] (1.73,-1.5) -- (0,-0.5) node[]{};
\draw[-] (-1.73,-1.5) -- (0,-0.5) node[]{};
\draw[-] (0,1.5) -- (0,-0.5) node[]{};
\draw[-] (0,0.5) -- (-0.865,-1) node[]{};
\draw[-] (0,0.5) -- (0.865,-1) node[]{};
\draw[-] (0.865,-1) -- (-0.865,-1) node[]{};
\draw[-] (0,1.5) -- (-0.865,-1) node[]{};
\draw[-] (-1.73,-1.5) -- (0.865,-1) node[]{};
\draw[-,red, thick] (0,1.5) -- (0.865,-1) node[]{};
\draw[-,green, thick] (0,0.5) -- (1.73,-1.5) node[]{};
\fill (1.73,-1.5) circle (0pt) node[right=2pt] {\footnotesize{$2$}};
\fill (-1.73,-1.5) circle (0pt) node[left=2pt] {\footnotesize{$1$}};
\fill (-0.865,-1) circle (0pt) node[above=1pt] {\footnotesize{$4$}};
\fill (0,1.5) circle (0pt) node[right=1pt] {\footnotesize{$3$}};
\fill (0,-0.5) circle (0pt) node[left=1pt] {\footnotesize{$0$}};
\fill (0,0.5) circle (0pt) node[below=1pt] {\footnotesize{$6$}};
\fill (0.865,-1) circle (0pt) node[above=3pt] {\footnotesize{$5$}};
\end{tikzpicture}\]

To be explicit, $\Cone(x_2, x_6) \in \Sigma_1$ (green line), while $\Cone(x_3, x_5) \in \Sigma_2$ (red line). Let $X_1 = X(\Sigma_1)$ and $X_2 = X(\Sigma_2)$. Then $X_1$ is a proper non-projective toric variety of Picard rank $4$, and it flops to the projective toric variety $X_2$. $X_1$ has primitive relation $x_3 + x_5 = x_2 + x_6$ which corresponds to this flop.
Note that $\Sigma_1$ has nine primitive collections and $-K_{X_1}$ is nef.

We use \texttt{primitiveBlowup} to get a star subdivision of $X_1$ along $\Cone(x_2,x_6)$ with an extra ray $x_2 + x_6 = (-1, 0, 0)$ produced, then use \texttt{primitiveBlowdown} to blow down along the primitive collection $\{x_3, x_5\}$ in order to obtain a toric variety $X_1$ flops to.
We can check using the function \texttt{areIsomorphic} that the resulting toric variety is indeed isomorphic to $X_2$, which is projective and has seven primitive collections.
\end{example}

By Example~\ref{eg:4.PB-33flop} above, we observe that the number of primitive collections of a smooth projective toric variety may change under a flop in an interesting way. 
Let's first consider the case of a toric variety $X$ with three primitive collections.
In this case, each of the three primitive relations is contained in an extremal ray of the Mori cone $\overline{\NE}(X)$, and
it is in general easier to study the star subdivisions and blowdowns in terms of primitive collections for a splitting fan.
An immediate consequence from \cite[Theorem 4.3]{Sat:00} and \cite[Corollary 4.9]{Sat:00} gives the following results.

\begin{corollary}\label{cor:flop3PC}
    Given a smooth projective toric variety $X$ with Picard rank $3$.
    Let $\Sigma$ be the splitting fan of $X$ satisfying $\rays(X) = \bigsqcup_{i = 1}^3 P_i$ where $P_1 = \{x_1,\cdots,x_l\}$, $P_2$ and $P_3$ are disjoint primitive collections of $\Sigma$.
    Assume $X$ has a primitive relation $\sum_{i = 1}^l x_i = \sum_{i = 1}^l y_i$.
    If there exist $y_j \in P_2$, $y_k \in P_3$ for some $j \neq k$,
    then $X$ flops to a toric variety with five primitive collections, otherwise it flops to a toric variety with three primitive collections.
\end{corollary}
\begin{proof}
    We denote by $\Y = \{y_1,\cdots,y_l\}$, $\Y_2 = \{y_j \mid y_j \in P_2\}$ and $\Y_3 = \{y_k \mid y_k \in P_3\}$.
    Let $u = \sum_{i = 1}^l x_i = \sum_{i = 1}^l y_i$.
    If both $\Y_2$ and $\Y_3$ are nonempty, then after star subdivision of $X$ along $(\Cone(\Y),u)$, in addition to the primitive collections $P_1,P_2,P_3$ and $\Y$, there are new ones: $(P_2 \smallsetminus \Y_2) \cup \{u\}$ and $(P_3 \smallsetminus \Y_3) \cup \{u\}$. 
    A follow-up blowdown along $P_1$ decreases the number of primitive collections to five: $P_2$, $P_3$, $\Y$, $(P_2 \smallsetminus \Y_2) \cup P_1$ and $(P_2 \smallsetminus \Y_3) \cup P_1$.
    If one of $\Y_2, \Y_3$ is empty, say $\Y_2 = \varnothing$, then $\Y \subsetneqq P_3$ ($\Y \neq P_3$ since $\Y$ is a cone in $\Sigma$). 
    The primitive collections after star subdivision along $\Y$ are $\Y$, $P_1$, $P_2$ and $(P_3 \smallsetminus \Y) \cup \{u\}$.
    Followed by a blowdown along $P_1$ decreases the number of primitive collections by one: $\Y$, $P_2$ and $(P_3 \smallsetminus \Y) \cup P_1$ whose associated toric variety is a projective bundle.  
\end{proof}

If there is a flop of $X$ to itself, then we call such toric variety $X$ ``self-flopped''. 
We start with smooth toric projective bundles of Picard rank $2$.

\begin{proposition}\label{prop:2-selfflopped}
    Let $X = \PP_{\PP^r}(\cO \oplus \cO^{\oplus d'-2r-1} \oplus \cO(1)^{\oplus (r + 1)})$ be a $\PP^{d' - r}$-bundle over $\PP^r$ of dimension $d'$. Then $X$ admits a self-flop.
\end{proposition}
\begin{proof}
    Due to the primitive relation
    \[\sum_{i = 1}^{r + 1} y_i = \sum_{i = d' - 2r}^{d'-r} x_i\]
    the projective bundle $X$ admits a flop. To see that this flop maps $X$ to itself, we first blow up $X$ along the cone $\Cone(\{x_{d'-2r},\cdots, x_{d'-r}\})$ with new ray 
    \[x' = \sum_{i = 1}^{r + 1} y_i = \sum_{i = d'-2r}^{d'-r} x_i.\]
    By \cite[Theorem 4.3]{Sat:00}, the primitive collections after blowup become
    \begin{align*}
        & \X' = \{x_{d'-2r},\cdots, x_{d'-r}\},\quad \Y  = \{y_1,\cdots, y_{r+1}\}, \\
        & \X'' = \{x_1,\cdots, x_{d'-2r-1},x_{d'-r+1}, x'\}.
    \end{align*} 
    Then a blowdown along $\Cone(\{y_1,\cdots,y_{r+1}\})$, by \cite[Corollary 4.9]{Sat:00}, induces the resulting primitive collections 
    \[\X' = \{x_{d'-2r},\cdots, x_{d'-r}\}, \quad \Y' = \{x_1,\cdots, x_{d'-2r-1},x_{d'-r+1}\} \cup \Y,\]
    whose associated toric variety can be easily seen to be isomorphic to $X$.
\end{proof}

For toric varieties of Picard rank $3$ with three primitive collections discussed in Corollary~\ref{cor:flop3PC}, self-flop only occurs when either $\Y_2 = \varnothing$ or $\Y_3 = \varnothing$.

\begin{remark}
    Under the above assumptions with $\Y_2 = \varnothing$ thus $\Y \subsetneqq P_3$ so that $X$ flops to a smooth toric variety with three primitive collections. 
    Consider the following two cases: $\sum_{x \in P_2} x = 0$ and $\sum_{x \in P_2} x = \sum_{\rho \in \sigma(1)} x_{\rho}$ for some cone $\sigma \in \Sigma$.
    If it is the first case, then such flop is a self-flop. 
    For the second case, if there is a permutation $\tau$ on the set of rays $\Sigma(1)$ with $\tau(P_1) = \Y$, $\tau(\Y) = P_1$ and preserving other rays, that also preserves the set $\sigma(1)$, then such primitive relation $\sum_{i = 1}^l x_i = \sum_{i = 1}^l y_i$ provides a self-flop.
\end{remark}

For a toric variety $X$ with five primitive collections defined in Section~\ref{sec:batyrevconstruction} with primitive relations given by \eqref{eq:5primitivereln},
it is easy to see that the first, second and fourth primitive relations are contained in extremal rays of the Mori cone $\overline{\NE}(X)$.
Under the above assumption on $X$, a similar discussion in Corollary~\ref{cor:flop3PC} implies the following result.

\begin{corollary}\label{cor:flop5PC}
    Let $X$ be a toric variety defined as above. If $X$ satisfies one of the following three conditions
    \begin{enumerate}[label = (\arabic*)]
        \item All $c_i = 0$ for $2 \leq i \leq p_2$, $b_i = 0$ for $1 \leq i \leq p_3$, and $p_0 + p_1 = p_3$;
        \item $p_1 + p_2 = p_4$;
        \item $p_3 + p_4 = p_1$,
    \end{enumerate}
    then $X$ flops to a variety with three primitive collections.
    If there exists a nonempty proper subset $X_2' \subset X_2 \smallsetminus \{z_1\}$ with $|X_2'| = p_2'$ satisfying $p_2' + p_3 = p_0 + p_1$, 
    such that $c_i = 1$ for all $z_i \in X_2'$ and $c_j = 0$ for all $z_j \in X_2 \smallsetminus (X_2' \cup \{z_1\})$, and $b_i = 0$ for all $1 \leq i \leq p_3$,
    then $X$ flops to a variety with five primitive collections.
    In this case, if in addition, $p_0 = p_3$ (and thus $p_1 = p_2'$), then $X$ is self-flopped.
\end{corollary}

\section{A converse to Bott vanishing}


Assume $X$ is a (smooth, proper) weak Fano toric variety, so $-K_X$ is nef and big. 
By \cite[Theorem 9.3.3]{CLS:11}, we have $H^p(X,\Omega^1_X \otimes \cO(-K_X)) = 0$ for all $p > 1$. If $X$ is Fano, then $H^1(X,\Omega_X^1 \otimes \cO(-K_X)) = 0$ by the Bott-Steenbrink-Danilov vanishing theorem (see \cite[Theorem 9.3.1]{CLS:11}). In fact, the converse holds: this cohomology group vanishes precisely when $X$ is Fano. We discovered this interesting fact empirically through numerical experimentation with Macaulay2.

\begin{theorem}\label{thm:converseBott}
    Let $X$ be a weak Fano toric variety. If $H^1(X, \Omega_X^1 \otimes \cO(-K_X)) = 0$, then $X$ is Fano.
\end{theorem}
\begin{proof}
    Assume $X$ is a smooth projective toric variety of dimension $d$ which is weak Fano, and let $\Sigma \subseteq N_{\RR} = \RR^d$ be its associated complete fan.
    The anticanonical divisor $-K_X = \sum_{\rho \in \Sigma(1)} D_\rho$ is nef (and it is always big for toric varieties).
    By the first Ishida complex, the cotangent sheaf satisfies the exact sequence
    \[0 \to \Omega_X^1 \to M \otimes_\ZZ \cO_X \to \bigoplus_{\rho \in \Sigma(1)}\cO_{D_\rho} \to 0.\]
    Twisting the sequence by $\cO_X(-K_X)$ and taking long exact sequence on cohomology we have
    \begin{align*}
        0 & \to H^0(X,\Omega_X^1(-K_X)) \to H^0(X,M\otimes_\ZZ \cO_X(-K_X)) \stackrel{\phi}{\to} \bigoplus_{\rho \in \Sigma(1)}H^0(X,\cO_{D_{\rho}}(-K_X)) \\
        & \to H^1(X,\Omega_X^1(-K_X)) \to M \otimes_\ZZ H^1(X, \cO_X(-K_X)) = 0
    \end{align*}
    since $H^1(X, \cO_X(-K_X)) = 0$ by Demazure vanishing. 
    Let's first consider $H^0(X,\cO_{D_\rho}(-K_X))$ for any fixed ray $\rho \in \Sigma(1)$. 
    By the exact sequence $0 \to \cO_X(-D_\rho) \to \cO_X \to \cO_{D_\rho} \to 0$ we have
    \begin{align*}
        0 & \to H^0(X,\cO(-K_X - D_{\rho})) \to H^0(X,\cO_X(-K_X)) \to H^0(X,\cO_{D_\rho}(-K_X)) \\ 
        & \to H^1(X, \cO_X(-K_X - D_{\rho})) \to 0.
    \end{align*}
    By \cite[Exercise 9.1.13]{CLS:11}, the graded piece for $m \in M$ satisfies the property that $H^0(X,\cO_X(-K_X - D_\rho))_m \neq 0$ implies that $H^1(X, \cO_X(-K_X - D_\rho))_m = 0$. 
    Recall that in \cite[Section 9.1]{CLS:11}, for any Weil divisor $D = \sum_{\rho \in \Sigma(1)} a_\rho D_{\rho}$, we denote by 
    \[V_{D,m} = \bigcup_{\sigma \in \Sigma} \Conv \left(u_\rho \mid \rho \in \sigma(1), \langle m,u_\rho \rangle < -a_\rho \right),\]
    where $u_\rho$ is the primitive generator of the ray $\rho$.
    In particular, using \cite[Theorem 9.1.3]{CLS:11} of computing sheaf cohomology via reduced cohomology, we derive the following cases for $m \in M$ and the corresponding $m$-graded pieces:
    \begin{enumerate}[label = (\roman*)]
        \item If $H^0(X,\cO_X(-K_X - D_\rho))_m \neq 0$. 
        This is equivalent to $V_{-K_X - D_\rho,m} = \varnothing$ which implies $V_{-K_X,m} = \varnothing$. 
        Thus $H^0(X,\cO_X(-K_X))_m \cong \CC$.
        Together with $H^1(X, \cO_X(-K_X - D_\rho))_m = 0$, we have $H^0(X, \cO_{D_\rho}(-K_X))_m = 0$.

        \item If $H^1(X,\cO_X(-K_X - D_\rho))_m \neq 0$.
        It implies that $H^0(X,\cO_X(-K_X - D_\rho))_m = 0$, which is equivalent to $V_{-K_X - D_\rho,m} \neq \varnothing$.
        Thus for such $m \in M$, the exact sequence
        \[0 \to H^0(X,\cO_X(-K_X))_m \to H^0(X,O_{D_\rho}(-K_X))_m \to H^1(X, \cO_X(-K_X - D_\rho))_m \to 0\]
        forces $H^0(X,\cO_{D_\rho}(-K_X))_m \cong H^1(X,\cO_X(-K_X - D_\rho))_m \cong \CC$.       
        
        \item If $H^0(X,\cO_X(-K_X - D_\rho))_m \cong H^1(X,\cO_X(-K_X - D_\rho))_m = 0$.
        Since $V_{-K_X - D_\rho,m} \neq \varnothing$, either there exists a ray $\rho' \neq \rho$ such that $\langle m,u_{\rho'} \rangle < -1$ or $\langle m,u_\rho \rangle < 0$. Therefore
        \[H^0(X, \cO_{D_\rho}(-K_X))_m = \begin{cases}
            \CC &\quad\text{if $\langle m,u_\rho \rangle = -1$ and $\langle m, u_{\rho'}\rangle \geq -1$ for all $\rho' \neq \rho$,} \\
            0 &\quad\text{otherwise.} 
        \end{cases}\]
    \end{enumerate}
    
    Now let's turn our attention to the case when $X$ is weak Fano but not Fano.
    By \cite[Lemma 6.1.13]{CLS:11}, it implies that there are two maximal cones $\sigma, \sigma' \in \Sigma(d)$ with intersection $\sigma \cap \sigma' \in \Sigma(d - 1)$ a wall, such that $m_\sigma = m_{\sigma'} \in M$. 
    Let $m = m_\sigma = m_{\sigma'}$ be the common Cartier data in $M$.
    Firstly $H^0(X,\cO_X(-K_X))_m \cong \CC$ since $m \in P_{-K_X}$.
    Let $(\sigma \cap \sigma')(1) = \{\rho_1,\cdots, \rho_{d - 1}\}$, and consider $\sigma(1) = (\sigma \cap \sigma')(1) \cup \{\rho_d\}$ and $\sigma'(1) = (\sigma \cap \sigma')(1) \cup \{\rho_{d + 1}\}$. 
    Then for any $i = 1,\cdots, d+1$, we have $\langle m, u_{\rho_i} \rangle = -1$ and thus $V_{-K_X - D_{\rho_i},m} \neq \varnothing$. 
    It means that $H^0(X,\cO_{X}(-K_X - D_{\rho_i}))_m = 0$. 
    That $H^0(X,\cO_X(-K_X))_m \cong \CC$ forces $H^1(X,\cO_X(-K_X - D_{\rho_i}))_m = 0$, and hence $H^0(X,\cO_{D_{\rho_i}}(-K_X))_m \cong \CC$ for all $i = 1,\cdots,d+1$.
    So that $\phi$ fails to be surjective on $m$-graded pieces for $m = m_\sigma = m_{\sigma'}$ which implies that $H^1(X,\Omega_X^1(-K_X))_m \neq 0$.
\end{proof}

In the case that $X$ is of Picard rank $3$, the proof can be simplified by the following fact. 

\begin{proposition}
Let $X = X(\Sigma)$ be a weak Fano toric variety of Picard rank $3$. Then 
\[H^i(X, \cO_X(-K_X - D_\rho)) = 0\]
for any $i \geq 1$ and ray $\rho \in \Sigma(1)$.  
\end{proposition}
\begin{proof}
    Let's consider separately two cases: $X$ is a Batyrev variety and $X$ is a projective bundle.
    We denote by $D = -K_X - D_\rho$.
    First assume that $X$ is a Batyrev variety with five primitive collections~\eqref{eq:5primitivereln} denoted by $P_0,\cdots,P_4$.
    It is easy to see that $\deg P_1$ and $\deg P_3$ cannot be zero simultaneously: if so, then $p_2 + p_3 = 0$ which is a contradiction.
    Thus either $\deg P_1 > 0$ or $\deg P_3 > 0$.
    To prove the proposition, by \cite[Theorem 2.4(ii)]{Mus:02}, it suffices to find a $\QQ$-divisor $E = \sum_j a_jD_j$ on $X$ where $a_j \in [0,1] \cap \QQ$ such that $D + E$ is $\QQ$-ample.
    If the primitive generator $u_\rho \in X_2$ or $X_3$, we can take $E = D_\rho + D_{v_1} + \frac{1}{2}D_{y_1}$ if $\deg P_3 > 0$, and $E = D_\rho + D_{v_1} + \frac{1}{2}D_{u_1}$ if $\deg P_1 > 0$. 
    If $u_\rho \in X_0$, we take $E = D_\rho + \frac{1}{2}D_{y_1} + \frac{1}{4}D_{u_1}$ if $\deg P_3 > 0$ and $E = D_\rho + \frac{1}{2}D_{y_1} + \frac{3}{4}D_{u_1}$ if $\deg P_1 > 0$.
    If $u_\rho \in X_1$, we take $E = D_\rho + D_{v_1} + D_{z_1}$ if $\deg P_3 > 0$ and $E = D_\rho + D_{v_1} + \frac{1}{2}D_{u_1}$ if $\deg P_1 > 0$.
    If $u_\rho \in X_4$, we take $E = D_\rho + D_{v_1} + D_{z_1}$ if $\deg P_3 > 0$ and $E = D_\rho + D_{v_1} + \frac{1}{2(1 + b_i)}D_{t_i}$ if $\deg P_1 > 0$ for any $t_i \in X_3$.
    Now assume that $X$ is a projective bundle with three primitive collections stated in Proposition \ref{prop:3-primitiverelns}. 
    If $u_\rho \in \Z$, we take $E = D_\rho + D_{y_1} + D_{x_{d' - r + 1}}$. 
    Likewise if $u_\rho \in \Y$, we take $E = D_\rho + D_{y} + D_{x_{d' - r + 1}}$ for any $y \in \Y \smallsetminus \{u_\rho\}$.
    Finally if $u_\rho \in \X$, we take $E = D_{\rho} + D_{y_1} + \frac{1}{a_{d'-r} + 1}D_{x_1}$ if $u_\rho = x_{d' - r + 1}$ and $E = D_\rho + D_{y_1} + D_{x_{d' - r + 1}}$ otherwise.
    It can be easily checked that by choosing $E$ as above, $D + E$ can be made $\QQ$-ample. Therefore we obtain the vanishing of $H^i(X, \cO_X(-K_X - D_\rho))$ for any $i > 0$ and $\rho \in \Sigma(1)$.
\end{proof}

\begin{remark}
We can also derive Theorem \ref{thm:converseBott} from a combinatorial formula of Mavlyutov.
Notice that the anticanonical divisor $-K_X$ of $X$ is $d$-semiample in the sense of \cite{Mav:08}.
According to \cite[Corollary 2.7]{Mav:08}, we have
\begin{align}
    \dim H^1(X, \Omega_X^1 \otimes \cO_X(-K_X)) &= \sum_{\Gamma \preccurlyeq P_{-K_X}}l^*(\Gamma)\sum_{j = 0}^1{d - \dim \Gamma - j \choose 1 - j} (-1)^{1 - j}\left|\Sigma_{\sigma_\Gamma(j)}(1)\right| \notag \\
    &= \sum_{\Gamma \preccurlyeq P_{-K_X}} l^*(\Gamma)(\left|\Sigma_{\sigma_\Gamma}(1)\right| - d + \dim \Gamma) 
\end{align}
where the summation is over all faces $\Gamma$ of the polytope $P_{-K_X}$, and 
\begin{itemize}
    \item $l^*(\Gamma)$ is the number of interior lattice points inside the face $\Gamma$. Note that the lattice $N$ is equipped with the subspace topology from the standard topology on the Euclidean space $N_\RR$. Thus if $\Gamma \preccurlyeq P_{-K_X}$ is a vertex, $l^*(\Gamma) = 1$;
    \item $\sigma_\Gamma \subseteq N_\RR$ is the cone in the fan $\Sigma_{-K_X}$ which corresponds to the face $\Gamma$. Explicitly by duality, the dual cone $\sigma_\Gamma$ is defined as
    \[\sigma_\Gamma = \Cone(\{u_F \mid \text{ $F$ is any facet of $P_{-K_X}$ which contains $\Gamma$}\}),\]
    and $\Sigma_{\sigma_\Gamma} = \{\tau \in \Sigma \mid \widetilde{\pi}(\tau) \subseteq \sigma_\Gamma\}$ is a subfan of $\Sigma$;
    \item $\widetilde{\pi} : N_\RR \to N_{-K_X} \otimes_\ZZ \RR$ is the map defined in \cite[Theorem 2.2]{Mav:08} which maps the fan $\Sigma$ to the induced fan $\Sigma_{-K_X}$. Since the anticanonical divisor $-K_X$ on a toric variety is always big, the surjective homomorphism $\widetilde{\pi}$ is also injective. 
\end{itemize}
If $X$ is not Fano, \ie $-K_X$ is nef but not ample, then the support function $\varphi_{-K_X}$ is only convex but not strictly convex.
As in the proof of Theorem~\ref{thm:converseBott}, this is equivalent to the existence of two maximal cones $\sigma, \sigma' \in \Sigma(d)$ such that their Cartier data satisfy $m_\sigma = m_{\sigma'}$. 
Since $-K_X$ is basepoint-free there is a vertex $v \in M$ of the polytope $P_{-K_X}$ with
\[\left|\Sigma_{\sigma_v}(1)\right| - d + \dim v = \left|\Sigma_{\sigma_v}(1)\right| - d > 0,\]
thus $\dim H^1(X, \Omega_X^1(-K_X)) > 0$.
\end{remark}

\bibliographystyle{alpha}

\appendix

\section{Table of toric weak Fano threefolds of $\rho = 2$} \label{app:2Picard3Fold}

In the following table we give a complete list of weak Fano toric threefolds of Picard rank $2$ tabulating some numerical invariants (Chern numbers, space of vector fields). Note that they all have two primitive collections. They are all projective bundles over a projective space.

\begin{table}[h!]
\begin{minipage}{\textwidth}     
\centering\small
\begin{tabularx}{\textwidth}{|c|c|c|c|X|}
    \hline
	Fanography\footnote{For Fano varieties, this is the ID from Fanography \cite{fanography}.}  & \texttt{kleinschmidt}\footnote{This is the input to construct the variety with the \texttt{kleinschmidt} function in \texttt{NormalToricVarieties}.} & $c_1^3$ & $h^0(X, T_X)$ & Blowdown and flop \\ \hline
    2-34 & $(3, \{0\})$ & 54 & 11 & $\PP^1 \times \PP^2$ \\ \hline
    2-35 & $(3, \{1\})$ &  56 & 12 & Blowup of $\PP^3$ at a point \\ \hline
    2-36 & $(3, \{2\})$ &  62 & 15 & Blowup of $\PP(1,1,1,2)$ (weak Fano) at a point \\ \hline
    2-33 & $(3, \{0,1\})$ & 54 & 11 & Blowup of $\PP^3$ along a line \\ \hline\hline
    & $(3, \{3\})$ & 72 & 19 & Blowup of $\PP(1,1,1,3)$ (Fano) at a point \\ \hline
    & $(3, \{0,2\})$ & 54 & 13 & Blowup of $\PP(1,1,2,2)$ (Fano) along a line \\ \hline
    & $(3, \{1,1\})$ & 54 & 11 & Self-flopped \\ \hline
\end{tabularx}
\end{minipage}
\end{table}

\section{Table of toric weak Fano fourfolds of $\rho = 2$}\label{app:2Picard4Fold}

In the following table we give a complete list of toric weak Fano fourfolds of Picard rank $2$ tabulating some numerical invariants. Again, they all have $2$ primitive collections, and projective bundle structure.

\begin{table}[h!]
\begin{minipage}{\textwidth}     
\centering\small
\begin{tabularx}{\textwidth}{|c|c|c|c|X|}
    \hline
	Batyrev/Sato \footnote{For Fano varieties, this is the ID from Batyrev or Sato's list \cite{Bat:99, Sat:thesis}.} & \texttt{kleinschmidt}  & $c_1^4$ & $c_1^2c_2$ & Blowdown and flop \\ \hline
    $B_4$ & $(4, \{0\})$ & 512 & 224 & $\PP^1 \times \PP^3$ \\ \hline
    $B_3$ & $(4, \{1\})$ & 544 & 232 & Blowup of $\PP^4$ at a point \\ \hline
    $B_2$ & $(4, \{2\})$ & 640 & 256 & Blowup of $\PP(1,1,1,1,2)$ (Fano) at a point \\ \hline
    $B_1$ & $(4, \{3\})$ & 800 & 296 & Blowup of $\PP(1,1,1,1,3)$ (weak Fano) at a point \\ \hline
    $C_4$ & $(4, \{0,0\})$ & 486 & 216 & $\PP^2 \times \PP^2$ \\ \hline
    $C_2$ & $(4, \{0,1\})$ & 513 & 222 & Blowup of $\PP^4$ along a line \\ \hline
    $C_1$ & $(4, \{0,2\})$ & 594 & 240 & Blowup of $\PP(1,1,1,2,2)$ (weak Fano) along a line \\ \hline
    $C_3$ & $(4, \{1,1\})$ & 513 & 222 & Flopped to $\PP_{\PP^1}(\cO \oplus \cO(1)^{\oplus 3})$ (not weak Fano) \\ \hline
    $B_5$ & $(4, \{0,0,1\})$ & 512 & 224 & Blowup $\PP^4$ along a surface $\PP^2$ \\ \hline \hline
    & $(4, \{4\})$ & 1024 & 352 & Blowup of $\PP(1,1,1,1,4)$ (Fano) at a point \\ \hline
    & $(4, \{0,3\})$ & 729 & 270 & Blowup of $\PP(1,1,1,3,3)$ (Fano) along a line \\ \hline
    & $(4, \{1,2\})$ & 567 & 234 &  \\ \hline
    & $(4, \{0,0,2\})$ & 512 & 224 & Blowup of $\PP(1,1,2,2,2)$ (Fano) along a surface $\PP^2$ \\ \hline
    & $(4, \{0,1,1\})$ & 512 & 224 & Self-flopped \\ 
    \hline
\end{tabularx}
\end{minipage}
\end{table}

\newpage
\section{Table of toric weak Fano threefolds of $\rho = 3$}\label{app:3Picard3Fold}

In the following table we give a complete list of weak Fano toric threefolds of Picard rank $3$ tabulating some numerical invariants. 
We use the notations 3.PB and 3.BC for the varieties that are projective bundles and the result of the Batyrev construction, respectively.

\begin{table}[h!]
    \begin{minipage}{\textwidth}  
    \centering\small
    \begin{tabularx}{\textwidth}{|c|c|c|c|X|}
    \hline
	    ID (Fanography) & \texttt{projectiveBundleConstructor}  \footnote{\texttt{projectiveBundleConstructor} and \texttt{batyrevConstructor} can be found in the code repository associated to this paper \cite{weakFano-repo}.} & $c_1^3$ & $h^0(X, T_X)$ & Blowdown and flop \\ \hline

    3.PB-1 (3-28) & $(2, \{0\},\ \{\{0, -1\}\})$ & 48 & 9 & $\PP^1 \times \FF_1$; Blowup $\PP^1 \times \PP^2$ along a line \\ \hline
    3.PB-2 (3-27) & $(2, \{0\},\ \{\{0, 0\}\})$ & 48 & 9 & $\PP^1 \times \PP^1 \times \PP^1$ \\ \hline
    3.PB-3 (3-25) & $(2, \{0\},\ \{\{1, -1\}\})$ & 44 & 7 & Blowup $X_3(0,1)$ along a line \\ \hline
    3.PB-4 (3-31) & $(2, \{0\},\ \{\{1, 1\}\})$ & 52 & 11 & Blowup $X_3(1,1)$ along a line \\ \hline
    3.PB-5 (3-30) & $(2, \{1\},\ \{\{1, 0\}\})$ & 50 & 10 & Blowup $X_3(0,1)$ or $\text{Bl}_{p} \PP^3$ along a line \\ \hline \hline
    3.PB-6 & $(2, \{0\},\ \{\{0, -2\}\})$ & 48 & 10 & $\PP^1 \times \FF_2$ \\ \hline
    3.PB-7 & $(2, \{0\},\ \{\{1, -2\}\})$ & 40 & 7 & Blowup $X_3(0,2)$ along a line \\ \hline
    3.PB-8 & $(2, \{0\},\ \{\{1, 2\}\})$ & 56 & 13 &  \\ \hline
    3.PB-9 & $(2, \{0\},\ \{\{2, -2\}\})$ & 32 & 7 &  \\ \hline
    3.PB-10 & $(2, \{0\},\ \{\{2, 2\}\})$ & 64 & 16 &  \\ \hline
    3.PB-11 & $(2, \{1\},\ \{\{0, -1\}\})$ & 48 & 9 & Flopped to 3.BC-7 \\ \hline
    3.PB-12 & $(2, \{1\},\ \{\{1, -1\}\})$ & 46 & 8 & Blowup $X_3(1,1)$ along a line; Flopped to 3.BC-3 \\ \hline
    3.PB-13 & $(2, \{1\},\ \{\{1, 1\}\})$ & 54 & 12 & Flopped to 3.BC-8 \\ \hline
    3.PB-14 & $(2, \{1\},\ \{\{2, -1\}\})$ & 48 & 10 & Flopped to 3.BC-4 \\ \hline
    3.PB-15 & $(2, \{1\},\ \{\{2, 0\}\})$ & 56 & 13 & Blowup $X_3(2)$ along a line \\ \hline
    3.PB-16 & $(2, \{1\},\ \{\{2, 1\}\})$ & 64 & 16 & Flopped to 3.BC-9 \\ \hline
    3.PB-17 & $(2, \{2\},\ \{\{1, 0\}\})$ & 52 & 12 & Blowup $X_3(0,2)$ along a line \\ \hline
    3.PB-18 & $(2, \{2\},\ \{\{2, 0\}\})$ & 64 & 17 &  \\ 
    \hline    \hline
    ID (Fanography) & \texttt{batyrevConstructor} & $c_1^3$ & $h^0(X, T_X)$ & Blowdown and flop \\ \hline
    3.BC-1 (3-26) & $(\{1,1,2,1,1\},\{0\},\{0\})$ & 46 & 8 & Blowup $X_3(0,1)$ at a point, or $\PP^1 \times \PP^2$ along a line \\ \hline
    3.BC-2 (3-29) & $(\{2,1,1,1,1\},\{1\},\{\})$ & 50 & 10 &  Blowup $\text{Bl}_{p} \PP^3$ or $X_3(2)$ along a line \\ \hline \hline
    3.BC-3 & $(\{1, 1, 1, 1, 2\}, \{0\}, \{\})$ & 46 & 8 & Blowup $\PP^1 \times \PP^2$ at a point, or $X_3(1,1)$ along a line; Flopped to 3.PB-12 \\ \hline
    3.BC-4 & $(\{1, 1, 1, 1, 2\}, \{1\}, \{\})$ & 48 & 10 & Blowup $\text{Bl}_{p} \PP^3$ at a point; Flopped to 3.PB-14 \\ \hline
    3.BC-5 & $(\{1, 1, 2, 1, 1\}, \{0\}, \{1\})$ & 46 & 8 & Blowup $X_3(1,1)$ at a point, or $X_3(0,1)$ along a line; Self-flopped \\ \hline
    3.BC-6 & $(\{1, 1, 2, 1, 1\}, \{1\}, \{0\})$ & 46 & 10 & Blowup $X_3(0,2)$ at a point, or $X_3(0,1)$ along a line \\ \hline
    3.BC-7 & $(\{1, 2, 1, 1, 1\}, \{0\}, \{\})$ & 48 & 9 & Blowup $\text{Bl}_{p} \PP^3$ at a point; Flopped to 3.PB-11 \\ \hline
    3.BC-8 & $(\{1, 2, 1, 1, 1\}, \{1\}, \{\})$ & 54 & 12 & Blowup $X_3(2)$ at a point; Flopped to 3.PB-13 \\ \hline
    3.BC-9 & $(\{1, 2, 1, 1, 1\}, \{2\}, \{\})$ & 64 & 16 & Blowup $X_3(0,1)$ at a point; Flopped to 3.PB-16 \\ \hline
    3.BC-10 & $(\{2, 1, 1, 1, 1\}, \{2\}, \{\})$ & 58 & 13 & Blowup $X_3(2)$ or $X_3(0,1)$ along a line \\ \hline

\end{tabularx}
\end{minipage}
\end{table}

\newpage

\section{Table of toric weak Fano fourfolds of $\rho = 3$}\label{app:3Picard4Fold}

In the following table we give a complete list of weak Fano toric fourfolds of Picard rank $3$ tabulating some numerical invariants. 
We use the notations 4.PB and 4.BC for the varieties that are projective bundles and the result of the Batyrev construction, respectively.


\begin{table}[h!]
\begin{minipage}{\textwidth}  
    \centering\small
    \begin{tabularx}{\textwidth}{|c|c|c|c|X|}
    \hline
    ID (Batyrev/Sato\footnote{For Fano varieties, this is the ID from Batyrev or Sato's list \cite{Bat:99, Sat:thesis}.}) & \texttt{projectiveBundleConstructor}\footnote{\texttt{projectiveBundleConstructor} can be found in the code repository associated to this paper \cite{weakFano-repo}.}  & $c_1^4$ & $c_1^2c_2$ & Blowdown and flop \\ \hline
    4.PB-1 ($D_{14}$) & $(2, \{0\}, \{\{0, -1\}, \{0, -1\}\})$ & 432 & 204 & Blowup $X_4(0)$ along a surface \\ \hline
    4.PB-2 ($D_{13}$) & $(2, \{0\}, \{\{0, 0\}, \{0, 0\}\})$ & 432 & 204 & $\PP^1 \times \PP^1 \times \PP^2$; \\ \hline
    4.PB-3 ($D_{17}$) & $(2, \{0\}, \{\{0, -1\}, \{1, -1\}\})$ & 405 & 198 & Blowup $X_4(0,0,1)$ along a surface \\ \hline
    4.PB-4 ($D_7$) & $(2, \{0\}, \{\{0, 0\},\{1, 1\}\})$ & 486 & 216 & Blowup $X_4(0,1,1)$ along a surface \\ \hline
    4.PB-5 ($D_{15}$) & $(2, \{1\}, \{\{0, 0\},\{0, 0\}\})$ & 432 & 204 & $\FF_1 \times \PP^2$; Blowup $X_4(0,0)$ along a surface \\ \hline
    4.PB-6 ($D_{11}$) & $(2, \{1\}, \{\{0, 0\},\{1, 0\}\})$ & 459 & 210 & Blowup $X_4(0,0,1)$ along a surface \\ \hline
    4.PB-7 ($D_5$) & $(3, \{0\}, \{\{0, -2\}\})$ & 496 & 220 &  \\ \hline
    4.PB-8 ($D_{12}$) & $(3, \{0\}, \{\{0, -1\}\})$ & 448 & 208 & Blowup $X_4(0)$ along a line \\ \hline
    4.PB-9 ($D_{18}$) & $(3, \{0\}, \{\{1, -2\}\})$ & 400 & 196 & Blowup $X_4(0,2)$ along a surface \\ \hline
    4.PB-10 ($D_{19}$) & $(3, \{0\}, \{\{1, -1\}\})$ & 400 & 196 & Blowup $X_4(0,1)$ along a surface, or $X_4(0,0,1)$ along a line \\ \hline
    4.PB-11 ($D_6$) & $(3, \{0\}, \{\{1, 1\}\})$ & 496 & 220 & Blowup $X_4(1,1)$ along a surface, or $X_4(1,1,1)$ (not weak Fano) along a line \\ \hline
    4.PB-12 ($D_1$) & $(3, \{0\}, \{\{1, 2\}\})$ & 592 & 244 & Blowup $X_4(2,2)$ (not weak Fano) along a surface \\ \hline
    4.PB-13 ($D_9$) & $(3, \{1\}, \{\{0, -1\}\})$ & 464 & 212 &  \\ \hline
    4.PB-14 ($D_{16}$) & $(3, \{1\}, \{\{1, -1\}\})$ & 432 & 204 & Blowup $X_4(1,1)$ along a surface \\ \hline
    4.PB-15 ($D_{8}$) & $(3, \{1\}, \{\{1, 0\}\})$ & 480 & 216 & Blowup $X_4(0,1)$ along a surface, or $X_4(1)$ along a line \\ \hline
    4.PB-16 ($D_3$) & $(3, \{1\}, \{\{1, 1\}\})$ & 560 & 236 & Blowup $X_4(1,2)$ along a surface \\ \hline
    4.PB-17 ($D_2$) & $(3, \{2\}, \{\{1, 0\}\})$ & 576 & 240 & Blowup $X_4(0,2)$ along a surface \\ \hline
    4.PB-18 ($D_{10}$) & $(3, \{0, 1\}, \{\{1, 0\}\})$ & 464 & 212 & Blowup $X_4(0,0,1)$ along a line, or $X_4(1)$ along a surface \\ \hline
    4.PB-19 ($D_4$) & $(3, \{0, 1\}, \{\{2, 0\}\})$ & 560 & 236 & Blowup $X_4(2)$ along a surface \\ \hline
    \hline
    4.PB-20 & $(2, \{0\},\{\{0, -2\}, \{0, -2\}\})$ & 432 & 204 &  \\ \hline
    4.PB-21 & $(2, \{0\},\{\{0, -1\}, \{0, 0\}\})$ & 432 & 204 & Self-flopped \\ \hline
    4.PB-22 & $(2, \{0\},\{\{0, -2\}, \{1, -2\}\})$ & 378 & 192 & Blowup $X_4(0,0,2)$ along a surface \\ \hline
    4.PB-23 & $(2, \{0\},\{\{0, -1\}, \{1, 0\}\})$ & 459 & 210 & Blowup $X_4(1,1,1)$ (not weak Fano) along a surface; Flop to 4.PB-39 \\ \hline
    4.PB-24 & $(2, \{0\},\{\{0, 0\}, \{1, -1\}\})$ & 378 & 192 & Blowup $X_4(0,1,1)$ along a surface; Self-flopped \\ \hline
    4.PB-25 & $(2, \{0\},\{\{0, 0\}, \{1, 2\}\})$ & 540 & 228 & Blowup $X_4(0,2,2)$ (not weak Fano) along a surface \\ \hline
    4.PB-26 & $(2, \{0\},\{\{0, -2\}, \{2, -2\}\})$ & 324 & 180 &  \\ \hline
    4.PB-27 & $(2, \{0\},\{\{0, -1\}, \{2, 0\}\})$ & 486 & 216 & Flop to 4.PB-44 \\ \hline
    4.PB-28 & $(2, \{0\},\{\{0, 0\}, \{2, -1\}\})$ & 324 & 180 & Self-flopped \\ \hline
    4.PB-29 & $(2, \{0\},\{\{0, 0\}, \{2, 2\}\})$ & 648 & 252 &  \\ \hline
    4.PB-30 & $(2, \{0\},\{\{1, -1\}, \{1, 0\}\})$ & 405 & 198 & Flop to 4.PB-38 \\ \hline
    \end{tabularx}
    \end{minipage}
\end{table}

\newpage

\begin{table}[h!]
\begin{minipage}{\textwidth}  
    \centering\small
    \begin{tabularx}{\textwidth}{|c|c|c|c|X|}
    \hline
    ID (Batyrev/Sato) & \texttt{projectiveBundleConstructor}  & $c_1^4$ & $c_1^2c_2$ & Blowdown and flop \\ \hline
    4.PB-31 & $(2, \{0\},\{\{1, 1\}, \{1, 1\}\})$ & 486 & 216 & Self-flopped \\ \hline
    4.PB-32 & $(2, \{1\},\{\{0, -1\}, \{0, -1\}\})$ & 432 & 204 & Flop to 4.BC-18 \\ \hline
    4.PB-33 & $(2, \{1\},\{\{0, -1\}, \{1, -1\}\})$ & 432 & 204 & Blowup $X_4(0,1,1)$ along a surface; Flop to 4.BC-19 \\ \hline
    4.PB-34 & $(2, \{1\},\{\{0, 0\}, \{1, 1\}\})$ & 513 & 222 & Blowup $X_4(0,1,2)$ (not weak Fano) along a surface; Flop to 4.BC-21 \\ \hline
    4.PB-35 & $(2, \{1\},\{\{0, -1\}, \{2, -1\}\})$ & 486 & 216 & Flop to 4.BC-20 \\ \hline
    4.PB-36 & $(2, \{1\},\{\{0, 0\}, \{2, 0\}\})$ & 540 & 228 & Blowup $X_4(0,2)$ along a surface \\ \hline
    4.PB-37 & $(2, \{1\},\{\{0, 0\}, \{2, 1\}\})$ & 648 & 252 & Flop to 4.BC-23 \\ \hline
    4.PB-38 & $(2, \{1\},\{\{1, -1\}, \{1, -1\}\})$ & 405 & 198 & Flop to 4.PB-30 or 4.BC-10 \\ \hline
    4.PB-39 & $(2, \{1\},\{\{1, 0\}, \{1, 0\}\})$ & 459 & 210 & Flop to 4.PB-23; Blowup $X_4(1,1)$ along a surface \\ \hline
    4.PB-40 & $(2, \{1\},\{\{1, 0\}, \{1, 1\}\})$ & 486 & 216 & Self-flopped; Flop to 4.BC-22 \\ \hline
    4.PB-41 & $(2, \{2\},\{\{0, 0\}, \{0, 0\}\})$ & 432 & 204 & $\FF_2 \times \PP^2$ \\ \hline
    4.PB-42 & $(2, \{2\},\{\{0, 0\}, \{1, 0\}\})$ & 486 & 216 & Blowup $X_4(0,0,2)$ along a surface \\ \hline
    4.PB-43 & $(2, \{2\},\{\{0, 0\}, \{2, 0\}\})$ & 648 & 252 &  \\ \hline
    4.PB-44 & $(2, \{2\},\{\{1, 0\}, \{1, 0\}\})$ & 486 & 216 & Flop to 4.PB-27 \\ \hline
    4.PB-45 & $(3, \{0\}, \{\{0, -3\}\})$ & 576 & 240 &  \\ \hline
    4.PB-46 & $(3, \{0\}, \{\{1, -3\}\})$ & 432 & 204 & Blowup $X_4(0,3)$ along a surface \\ \hline
    4.PB-47 & $(3, \{0\}, \{\{1, 3\}\})$ & 720 & 276 & Blowup $X_4(3,3)$ (not weak Fano) along a surface \\ \hline
    4.PB-48 & $(3, \{0\}, \{\{2, -3\}\})$ & 288 & 168 &  \\ \hline
    4.PB-49 & $(3, \{0\}, \{\{2, -2\}\})$ & 304 & 172 &  \\ \hline
    4.PB-50 & $(3, \{0\}, \{\{2, -1\}\})$ & 352 & 184 & Blowup $X_4(0,0,2)$ along a line \\ \hline
    4.PB-51 & $(3, \{0\}, \{\{2, 1\}\})$ & 544 & 232 & Blowup $X_4(2,2,2)$ (not weak Fano) along a line \\ \hline
    4.PB-52 & $(3, \{0\}, \{\{2, 2\}\})$ & 688 & 268 &  \\ \hline
    4.PB-53 & $(3, \{0\}, \{\{2, 3\}\})$ & 864 & 312 &  \\ \hline
    4.PB-54 & $(3, \{1\}, \{\{0, -2\}\})$ & 512 & 224 &  \\ \hline
    4.PB-55 & $(3, \{1\}, \{\{1, -2\}\})$ & 416 & 200 & Blowup $X_4(1,2)$ along a surface \\ \hline
    4.PB-56 & $(3, \{1\}, \{\{1, 2\}\})$ & 672 & 264 & Blowup $X_4(2,3)$ (not weak Fano) along a surface \\ \hline
    4.PB-57 & $(3, \{1\}, \{\{2, -2\}\})$ & 384 & 192 &  \\ \hline
    4.PB-58 & $(3, \{1\}, \{\{2, -1\}\})$ & 464 & 212 &  \\ \hline
    4.PB-59 & $(3, \{1\}, \{\{2, 0\}\})$ & 576 & 240 & Blowup $X_4(2)$ along a line \\ \hline
    4.PB-60 & $(3, \{1\}, \{\{2, 1\}\})$ & 720 & 276 &  \\ \hline
    4.PB-61 & $(3, \{1\}, \{\{2, 2\}\})$ & 896 & 320 &  \\ \hline
    4.PB-62 & $(3, \{2\}, \{\{1, -1\}\})$ & 512 & 224 & Blowup $X_4(1,2)$ along a surface \\ \hline
    4.PB-63 & $(3, \{2\}, \{\{1, 1\}\})$ & 672 & 264 & Blowup $X_4(1,3)$ (not weak Fano) along a surface \\ \hline
    4.PB-64 & $(3, \{2\}, \{\{2, -1\}\})$ & 672 & 264 &  \\ \hline
    4.PB-65 & $(3, \{2\}, \{\{2, 0\}\})$ & 816 & 300 &  \\ \hline
    4.PB-66 & $(3, \{2\}, \{\{2, 1\}\})$ & 992 & 344 &  \\ \hline
    4.PB-67 & $(3, \{3\}, \{\{1, 0\}\})$ & 720 & 276 & Blowup $X_4(0,3)$ along a surface \\ \hline
    4.PB-68 & $(3, \{3\}, \{\{2, 0\}\})$ & 1152 & 384 &  \\ \hline
    4.PB-69 & $(3, \{0, 1\}, \{\{1, -1\}\})$ & 416 & 200 & Blowup $X_4(0,1,1)$ along a line; Flop to 4.BC-13 \\ \hline
    4.PB-70 & $(3, \{0, 1\}, \{\{1, 1\}\})$ & 512 & 224 & Blowup $X_4(1,1,2)$ (not weak Fano) along a line; Flop to 4.BC-30 \\ \hline
    \end{tabularx}
    \end{minipage}
\end{table}

\newpage
\begin{table}[h!]
\begin{minipage}{\textwidth}  
    \centering\small
    \begin{tabularx}{\textwidth}{|c|c|c|c|X|}
    \hline
    ID (Batyrev/Sato) & \texttt{projectiveBundleConstructor} & $c_1^3$ & $c_1^2c_2$ & Blowdown and flop \\ \hline
    4.PB-71 & $(3, \{0, 1\}, \{\{2, -1\}\})$ & 464 & 212 & Flop to 4.BC-24 \\ \hline
    4.PB-72 & $(3, \{0, 1\}, \{\{2, 1\}\})$ & 656 & 260 & Flop to 4.BC-31 \\ \hline
    4.PB-73 & $(3, \{0, 1\}, \{\{3, -1\}\})$ & 576 & 240 & Flop to 4.BC-25 \\ \hline
    4.PB-74 & $(3, \{0, 1\}, \{\{3, 0\}\})$ & 720 & 276 & Blowup $X_4(3)$ along a surface \\ \hline
    4.PB-75 & $(3, \{0, 1\}, \{\{3, 1\}\})$ & 864 & 312 & Flop to 4.BC-32 \\ \hline
    4.PB-76 & $(3, \{0, 2\}, \{\{1, 0\}\})$ & 480 & 216 & Blowup $X_4(0,0,2)$ along a line \\ \hline
    4.PB-77 & $(3, \{0, 2\}, \{\{2, 0\}\})$ & 624 & 252 &  \\ \hline
    4.PB-78 & $(3, \{0, 2\}, \{\{3, 0\}\})$ & 864 & 312 &  \\ \hline
    4.PB-79 & $(3, \{1, 1\}, \{\{1, 0\}\})$ & 480 & 216 & Blowup $X_4(0,1,1)$ along a line; Self-flopped \\ \hline
    4.PB-80 & $(3, \{1, 1\}, \{\{2, 0\}\})$ & 624 & 252 & Self-flopped \\ \hline
    4.PB-81 & $(3, \{1, 1\}, \{\{3, 0\}\})$ & 864 & 312 & Self-flopped \\ \hline
    \hline
    ID (Batyrev/Sato) & \texttt{batyrevConstructor}\footnote{\texttt{batyrevConstructor} can be found in the code repository associated to this paper \cite{weakFano-repo}.} & $c_1^4$ & $c_1^2c_2$ & Blowdown and flop \\ \hline
    4.BC-1 ($G_5$) & $(\{1, 1, 2, 1, 2\},\{0\},\{0\})$ & 406 & 196 & Blowup $X_4(1,1)$ along a surface, or $X_4(0,0)$ along a line  \\ \hline
    4.BC-2 ($E_3$) & $(\{1, 1, 3, 1, 1\},\{0\},\{0, 0\})$ & 431 & 206 & Blowup $X_4(1)$ or $X_4(0)$ along a surface, or $X_4(0,0,1)$ at a point \\ \hline
    4.BC-3 ($G_3$) & $(\{1, 2, 1, 1, 2\},\{0\},\{\})$ & 433 & 202 & Blowup $X_4(1,1)$ along a line \\ \hline
    4.BC-4 ($G_1$) & $(\{1, 2, 1, 1, 2\},\{1\},\{\})$ & 529 & 226 &  \\ \hline
    4.BC-5 ($G_6$) & $(\{2, 1, 2, 1, 1\},\{0\},\{0\})$ & 401 & 194 & Blowup $X_4(0,1)$ along a line, or $X_4(0,0)$ along a surface \\ \hline
    4.BC-6 ($G_4$) & $(\{2, 1, 2, 1, 1\},\{0\},\{1\})$ & 417 & 198 & Blowup $X_4(1,1)$ along a line, or $X_4(0,1)$ along a surface \\ \hline
    4.BC-7 ($G_2$) & $(\{2, 1, 2, 1, 1\},\{1\},\{0\})$ & 450 & 204 & Blowup $X_4(0,2)$ along a line, or $X_4(0,1)$ along a surface \\ \hline
    4.BC-8 ($E_2$) & $(\{3, 1, 1, 1, 1\},\{1\},\{\})$ & 489 & 222 & Blowup $X_4(2)$ or $X_4(1)$ along a surface \\ \hline
    4.BC-9 ($E_1$) & $(\{3, 1, 1, 1, 1\},\{2\},\{\})$ & 605 & 254 & Blowup $X_4(3)$ or $X_4(2)$ along a surface \\ \hline
    \hline
    4.BC-10 & $(\{1, 1, 1, 2, 2\},\{0, 0\},\{\})$ & 405 & 198 & Flop to 4.PB-38; Blowup $X_4(0,0)$ at a point \\ \hline
    4.BC-11 & $(\{1, 1, 2, 1, 2\},\{0\},\{1\})$ & 433 & 202 & Self-flopped; Blowup $X_4(0,1)$ along a line \\ \hline
    4.BC-12 & $(\{1, 1, 2, 1, 2\},\{1\},\{0\})$& 433 & 202 & Blowup $X_4(0,1)$ along a line \\ \hline
    4.BC-13 & $(\{1, 1, 2, 2, 1\},\{0, 0\},\{0\})$ & 416 & 200 & Flop to 4.PB-69; Blowup $X_4(0,1,1)$ or $X_4(0)$ along a line \\ \hline
    4.BC-14 & $(\{1, 1, 3, 1, 1\},\{0\},\{0, 1\})$ & 431 & 206 & Self-flopped; Blowup $X_4(0,1,1)$ at a point, or $X_4(0,0,1)$ along a surface \\ \hline
    4.BC-15 & $(\{1, 1, 3, 1, 1\},\{1\},\{0, 0\})$ & 431 & 206 & Blowup $X_4(0,0,2)$ at a point, or $X_4(0,0,1)$ along a surface \\ \hline
    4.BC-16 & $(\{1, 2, 1, 1, 2\},\{2\},\{\})$ & 721 & 274 &  \\ \hline
    4.BC-17 & $(\{1, 2, 1, 2, 1\},\{0, 1\},\{\})$ & 487 & 214 & Blowup $X_4(1,2)$ along a line \\ \hline
    4.BC-18 & $(\{1, 2, 2, 1, 1\},\{0\},\{0\})$ & 432 & 204 & Blowup $X_4(1)$ along a line, or $X_4(0,1)$ at a point; Flop to 4.PB-32 \\ \hline
    4.BC-19 & $(\{1, 2, 2, 1, 1\},\{0\},\{1\})$ & 432 & 204 & Blowup $X_4(1,1)$ at a point; Flop to 4.PB-33 \\ \hline
    4.BC-20 & $(\{1, 2, 2, 1, 1\},\{0\},\{2\})$ & 486 & 216 & Blowup $X_4(1,2)$ at a point; Flop to 4.PB-35 \\ \hline
    \end{tabularx}
    \end{minipage}
\end{table}

\begin{table}[h!]
    \begin{minipage}{\textwidth}  
    \centering\small
    \begin{tabularx}{\textwidth}{|c|c|c|c|X|}
    \hline
    ID (Batyrev/Sato) & \texttt{batyrevConstructor} & $c_1^4$ & $c_1^2c_2$ & Blowdown and flop \\ \hline
    4.BC-21 & $(\{1, 2, 2, 1, 1\},\{1\},\{0\})$ & 513 & 222 & Blowup $X_4(0,2)$ at a point; Flop to 4.PB-34 \\ \hline
    4.BC-22 & $(\{1, 2, 2, 1, 1\},\{1\},\{1\})$ & 486 & 216 & Blowup $X_4(1,2)$ at a point; Flop to 4.PB-40 \\ \hline
    4.BC-23 & $(\{1, 2, 2, 1, 1\},\{2\},\{0\})$ & 648 & 252 & Blowup $X_4(0,3)$ at a point; Flop to 4.PB-37 \\ \hline
    4.BC-24 & $(\{2, 1, 1, 1, 2\},\{1\},\{\})$ & 464 & 212 & Flop to 4.PB-71; Blowup $X_4(1)$ along a line \\ \hline
    4.BC-25 & $(\{2, 1, 1, 1, 2\},\{2\},\{\})$ & 576 & 240 & Flop to 4.PB-73; Blowup $X_4(2)$ along a line \\ \hline
    4.BC-26 & $(\{2, 1, 1, 2, 1\},\{0, 1\},\{\})$ & 449 & 206 & Blowup $X_4(1,2)$ along a surface, or $X_4(0,1)$ along a line \\ \hline
    4.BC-27 & $(\{2, 1, 2, 1, 1\},\{0\},\{2\})$ & 487 & 214 & Blowup $X_4(1,2)$ along a line, or $X_4(0,2)$ along a surface \\ \hline
    4.BC-28 & $(\{2, 1, 2, 1, 1\},\{1\},\{1\})$ & 439 & 202 & Blowup $X_4(1,2)$ along a line, or $X_4(1,1)$ along a surface \\ \hline
    4.BC-29 & $(\{2, 1, 2, 1, 1\},\{2\},\{0\})$ & 553 & 226 & Blowup $X_4(0,3)$ along a line, or $X_4(0,2)$ along a surface \\ \hline
    4.BC-30 & $(\{2, 2, 1, 1, 1\},\{1\},\{\})$ & 512 & 224 & Blowup $X_4(2)$ along a line; Flop to 4.PB-70 \\ \hline
    4.BC-31 & $(\{2, 2, 1, 1, 1\},\{2\},\{\})$ & 656 & 260 & Blowup $X_4(3)$ along a line; Flop to 4.PB-72 \\ \hline
    4.BC-32 & $(\{2, 2, 1, 1, 1\},\{3\},\{\})$ & 864 & 312 & Blowup $X_4(4)$ along a line; Flop to 4.PB-75 \\ \hline
    4.BC-33 & $(\{3, 1, 1, 1, 1\},\{3\},\{\})$ & 779 & 302 & Blowup $X_4(4)$ or $X_4(3)$ along a surface \\ 
    \hline    
    \end{tabularx}
    \end{minipage}
\end{table}

\newpage

\end{document}